%% file: HA_3.1.2.3.tex
\pdfoutput=1
\documentclass[a4paper]{amsart}
\usepackage[T1]{fontenc}
\usepackage{amstext}
\usepackage{amsthm}
\usepackage{amssymb}
\usepackage[unicode=true,pdfusetitle,
 bookmarks=true,bookmarksnumbered=false,bookmarksopen=false,
 breaklinks=false,pdfborder={0 0 0},pdfborderstyle={},backref=false,colorlinks=false]
 {hyperref}

\makeatletter

\numberwithin{equation}{section}
\numberwithin{figure}{section}
\theoremstyle{plain}
\newtheorem{thm}{\protect\theoremname}[section]
\theoremstyle{plain}
\newtheorem{lem}[thm]{\protect\lemmaname}
\theoremstyle{plain}
\newtheorem{prop}[thm]{\protect\propositionname}
\theoremstyle{remark}
\newtheorem{rem}[thm]{\protect\remarkname}
\theoremstyle{definition}
\newtheorem{notation}[thm]{\protect\notationname}
\theoremstyle{definition}
\newtheorem{defn}[thm]{\protect\definitionname}
\theoremstyle{plain}
\newtheorem{cor}[thm]{\protect\corollaryname}

\usepackage{tikz-cd}
\usepackage{mathtools}
\usepackage{adjustbox}
\usepackage{enumitem}
\tikzcdset{scale cd/.style={every label/.append style={scale=#1},
    cells={nodes={scale=#1}}}}
\usepackage{eucal}
\usepackage{url}
\usepackage{mleftright}
\mleftright

\makeatother

\providecommand{\corollaryname}{Corollary}
\providecommand{\definitionname}{Definition}
\providecommand{\lemmaname}{Lemma}
\providecommand{\notationname}{Notation}
\providecommand{\propositionname}{Proposition}
\providecommand{\remarkname}{Remark}
\providecommand{\theoremname}{Theorem}

\begin{document}
\global\long\def\sf#1{\mathsf{#1}}%

\global\long\def\scr#1{\mathscr{{#1}}}%

\global\long\def\cal#1{\mathcal{#1}}%

\global\long\def\bb#1{\mathbb{#1}}%

\global\long\def\bf#1{\mathbf{#1}}%

\global\long\def\frak#1{\mathfrak{#1}}%

\global\long\def\fr#1{\mathfrak{#1}}%

\global\long\def\u#1{\underline{#1}}%

\global\long\def\tild#1{\widetilde{#1}}%

\global\long\def\mrm#1{\mathrm{#1}}%

\global\long\def\pr#1{\left(#1\right)}%

\global\long\def\abs#1{\left|#1\right|}%

\global\long\def\inp#1{\left\langle #1\right\rangle }%

\global\long\def\br#1{\left\{  #1\right\}  }%

\global\long\def\norm#1{\left\Vert #1\right\Vert }%

\global\long\def\hat#1{\widehat{#1}}%

\global\long\def\opn#1{\operatorname{#1}}%

\global\long\def\bigmid{\,\middle|\,}%

\global\long\def\Top{\sf{Top}}%

\global\long\def\Set{\sf{Set}}%

\global\long\def\SS{\sf{sSet}}%

\global\long\def\Kan{\sf{Kan}}%

\global\long\def\Cat{\mathcal{C}\sf{at}}%

\global\long\def\imfld{\cal M\mathsf{fld}}%

\global\long\def\ids{\cal D\sf{isk}}%

\global\long\def\ich{\cal C\sf h}%

\global\long\def\SW{\mathcal{SW}}%

\global\long\def\SHC{\mathcal{SHC}}%

\global\long\def\B{\sf B}%

\global\long\def\Spaces{\sf{Spaces}}%

\global\long\def\Mod{\sf{Mod}}%

\global\long\def\Nec{\sf{Nec}}%

\global\long\def\Fin{\sf{Fin}}%

\global\long\def\Ch{\sf{Ch}}%

\global\long\def\Ab{\sf{Ab}}%

\global\long\def\SA{\sf{sAb}}%

\global\long\def\P{\mathsf{POp}}%

\global\long\def\Op{\mathcal{O}\mathsf{p}}%

\global\long\def\Opg{\mathcal{O}\mathsf{p}_{\infty}^{\mathrm{gn}}}%

\global\long\def\Tup{\mathsf{Tup}}%

\global\long\def\Del{\mathbf{\Delta}}%

\global\long\def\id{\mathrm{id}}%

\global\long\def\Aut{\operatorname{Aut}}%

\global\long\def\End{\operatorname{End}}%

\global\long\def\Hom{\operatorname{Hom}}%

\global\long\def\Ext{\operatorname{Ext}}%

\global\long\def\sk{\operatorname{sk}}%

\global\long\def\ihom{\underline{\operatorname{Hom}}}%

\global\long\def\N{\mathrm{N}}%

\global\long\def\-{\text{-}}%

\global\long\def\op{\mathrm{op}}%

\global\long\def\To{\Rightarrow}%

\global\long\def\rr{\rightrightarrows}%

\global\long\def\rl{\rightleftarrows}%

\global\long\def\mono{\rightarrowtail}%

\global\long\def\epi{\twoheadrightarrow}%

\global\long\def\comma{\downarrow}%

\global\long\def\ot{\leftarrow}%

\global\long\def\corr{\leftrightsquigarrow}%

\global\long\def\lim{\operatorname{lim}}%

\global\long\def\colim{\operatorname{colim}}%

\global\long\def\holim{\operatorname{holim}}%

\global\long\def\hocolim{\operatorname{hocolim}}%

\global\long\def\Ran{\operatorname{Ran}}%

\global\long\def\Lan{\operatorname{Lan}}%

\global\long\def\Sk{\operatorname{Sk}}%

\global\long\def\Sd{\operatorname{Sd}}%

\global\long\def\Ex{\operatorname{Ex}}%

\global\long\def\Cosk{\operatorname{Cosk}}%

\global\long\def\Sing{\operatorname{Sing}}%

\global\long\def\Sp{\operatorname{Sp}}%

\global\long\def\Spc{\operatorname{Spc}}%

\global\long\def\Ho{\operatorname{Ho}}%

\global\long\def\Fun{\operatorname{Fun}}%

\global\long\def\map{\operatorname{map}}%

\global\long\def\diag{\operatorname{diag}}%

\global\long\def\Gap{\operatorname{Gap}}%

\global\long\def\cc{\operatorname{cc}}%

\global\long\def\Ob{\operatorname{Ob}}%

\global\long\def\Map{\operatorname{Map}}%

\global\long\def\Rfib{\operatorname{RFib}}%

\global\long\def\Lfib{\operatorname{LFib}}%

\global\long\def\Tw{\operatorname{Tw}}%

\global\long\def\Equiv{\operatorname{Equiv}}%

\global\long\def\Arr{\operatorname{Arr}}%

\global\long\def\Cyl{\operatorname{Cyl}}%

\global\long\def\Path{\operatorname{Path}}%

\global\long\def\Alg{\operatorname{Alg}}%

\global\long\def\ho{\operatorname{ho}}%

\global\long\def\Comm{\operatorname{Comm}}%

\global\long\def\Triv{\operatorname{Triv}}%

\global\long\def\triv{\operatorname{triv}}%

\global\long\def\Env{\operatorname{Env}}%

\global\long\def\Act{\operatorname{Act}}%

\global\long\def\act{\operatorname{act}}%

\global\long\def\loc{\operatorname{loc}}%

\global\long\def\Assem{\operatorname{Assem}}%

\global\long\def\Nat{\operatorname{Nat}}%

\global\long\def\lax{\mathrm{lax}}%

\global\long\def\weq{\mathrm{weq}}%

\global\long\def\fib{\mathrm{fib}}%

\global\long\def\cof{\mathrm{cof}}%

\global\long\def\inj{\mathrm{inj}}%

\global\long\def\univ{\mathrm{univ}}%

\global\long\def\Ker{\opn{Ker}}%

\global\long\def\Coker{\opn{Coker}}%

\global\long\def\Im{\opn{Im}}%

\global\long\def\Coim{\opn{Im}}%

\global\long\def\coker{\opn{coker}}%

\global\long\def\im{\opn{\mathrm{im}}}%

\global\long\def\coim{\opn{coim}}%

\global\long\def\gn{\mathrm{gn}}%

\global\long\def\Mon{\mathrm{Mon}}%

\global\long\def\Un{\mathrm{Un}}%

\global\long\def\St{\mathrm{St}}%

\global\long\def\CA{\operatorname{CAlg}}%

\global\long\def\rd{\mathrm{rd}}%

\global\long\def\xmono#1#2{\stackrel[#2]{#1}{\rightarrowtail}}%

\global\long\def\xepi#1#2{\stackrel[#2]{#1}{\twoheadrightarrow}}%

\global\long\def\adj{\stackrel[\longleftarrow]{\longrightarrow}{\bot}}%

\global\long\def\btimes{\boxtimes}%

\global\long\def\ps#1#2{\prescript{}{#1}{#2}}%

\global\long\def\ups#1#2{\prescript{#1}{}{#2}}%

\global\long\def\hofib{\mathrm{hofib}}%

\global\long\def\cofib{\mathrm{cofib}}%

\global\long\def\Vee{\bigvee}%

\global\long\def\w{\wedge}%

\global\long\def\t{\otimes}%

\global\long\def\bp{\boxplus}%

\global\long\def\rcone{\triangleright}%

\global\long\def\lcone{\triangleleft}%

\global\long\def\S{\mathsection}%

\global\long\def\p{\prime}%

\global\long\def\pp{\prime\prime}%

\global\long\def\W{\overline{W}}%

\global\long\def\o#1{\overline{#1}}%

\title[Envelopes of Families of $\infty$-Operads and Operadic Kan Extensions]{Monoidal Envelopes of Families of $\infty$-Operads and $\infty$-Operadic Kan Extensions}
\author{Kensuke Arakawa}
\email{arakawa.kensuke.22c@st.kyoto-u.ac.jp}
\address{Department of Mathematics, Kyoto University, Kyoto, 606-8502, Japan}
\begin{abstract}
We provide details of the proof of Lurie's theorem on operadic Kan
extensions. Along the way, we generalize the construction of monoidal
envelopes of $\infty$-operads to families of $\infty$-operads and
use it to construct the fiberwise direct sum functor, both of which
we characterize by certain universal properties. Aside from their
use in elaborating the proof of Lurie's theorem, these results and
constructions have their independent interest.
\end{abstract}

\maketitle
\tableofcontents{}

\section*{Introduction}

In his monumental book \textit{Higher Algebra} \cite{HA}, Lurie introduced
(among other things) the notion of $\infty$-operads, an $\infty$-categorical
analog of (colored) operads. Given an $\infty$-operad $\cal O^{\t}$
and a symmetric monoidal $\infty$-category $\cal C^{\t}$, we can
form the $\infty$-category $\Alg_{\cal O}\pr{\cal C}$ of $\cal O$-algebras
in $\cal C$. If $f:\cal O^{\t}\to\cal O^{\p\t}$ is a morphism of
$\infty$-operads, pulling back along $f$ determines a functor $f^{*}:\Alg_{\cal O'}\pr{\cal C}\to\Alg_{\cal O}\pr{\cal C}$.
In analogy with restrictions and extensions of scalars, it is natural
to ask whether the map $f^{*}$ has a left adjoint. Lurie gives an
affirmative answer to this question in \cite[Corollary 3.1.3.5]{HA},
provided that $\cal C$ has a good supply of well-behaved colimits.
Because of its fundamental importance, the result has been applied
repeatedly in his work. 

Lurie's proof of \cite[Corollary 3.1.3.5]{HA} relies on another major
theorem on operadic Kan extensions \cite[Theorem 3.1.2.3]{HA} (or
Theorem \ref{thm:3.1.2.3}), which we call the \textbf{fundamental
theorem of operadic Kan extensions}, or FTOK for short. The proof
of FTOK, as Lurie himself acknowledges, is very long. Perhaps because
of the length of the proof, he left out some crucial details of the
arguments. This note aims to provide details of these nontrivial omissions.
(In particular, we do \textit{not} aim to provide a more concise proof
of FTOK than the one given in \cite{HA}.)

We can roughly divide the missing details into two parts: The first
one is the datum of ``coherent homotopy.'' In the proof of \cite[Theorem 3.1.2.3]{HA}
(to be more precise, on p. 337), Lurie claims that certain diagrams
are ``equivalent'' without writing the actual equivalence. Such
a practice is fairly common in the literature and is understandable
to some extent. However, in our case, we must be more attentive because
the details are somewhat involved. We thus give a complete treatment
of the equivalence in this note. (This corresponds to Lemma \ref{lem:contribution1}.)
The second missing detail is the verification that all the combinatorics
fit together. Such a detail, again, could be left to the reader if
the verification is trivial. However, we feel that this does not apply
to the case at hand; the details are complicated enough that it deserves
a separate treatment. We thus record every single detail of the verification.
(This corresponds to Lemmas \ref{lem:contribution2}, \ref{lem:contribution3},
and \ref{lem:contribution4}.) 

Here is an outline of this note. In Section \ref{sec:HA}, we quote
some results from Lurie's book \cite{HA} that we will use in this
paper, to draw a clear line between what we accept as established
and what we do not. Sections \ref{sec:families} through \ref{sec:direct_sum}
concern constructions on families of $\infty$-operads that we will
need in supplementing a part of Lurie's proof (namely, Lemma \ref{lem:contribution1}).
More precisely, in Section \ref{sec:families}, we will prove an equivalent
formulation of families of $\infty$-operads, and in Section \ref{sec:env},
we generalize Lurie's symmetric monoidal envelopes to families of
$\infty$-operads. Using symmetric monoidal envelopes, we can define
the fiberwise direct sum functor, whose properties we discuss in Section
\ref{sec:direct_sum}. In addition to its usage to fill in the details
of Lurie's argument, we believe that these constructions also have
some independent interest. In Section \ref{sec:FTOK}, we will reproduce
Lurie's proof of FTOK, and indicate the parts that require further
elaborations as lemmas. These lemmas will then be proved in Section
\ref{sec:leftover}, using results from earlier sections.

\section*{Notation and Terminology}

Our notation and terminology mostly follow those of \cite{HA}. Here
are some deviations.
\begin{itemize}
\item We will say that a morphism of simplicial sets is \textbf{final} if
it is cofinal in the sense of \cite{HTT}, and \textbf{initial} if
its opposite is final. 
\item If $\cal C$ is a category and $\alpha$ is an ordinal (or more generally
a well-ordered set), then an \textbf{$\alpha$-sequence} in $\cal C$
is a functor $F:\alpha\to\cal C$ such that the map $\colim_{\beta<\lambda}F\beta\to F\beta$
is an isomorphism for each limit ordinal less than $\alpha$. 
\item If $X$ is a simplicial set, then we denote the cone point of the
simplicial set $X^{\rcone}$ by $\infty$.
\item Let $p:X\to S$ be an inner fibration of simplicial sets, and let
$f:s\to s'$ be an edge of $S$. We say that $p$ \textbf{admits cocartesian
morphisms over }$f$ if for each object $x\in X\times_{S}\{s\}$,
there is a $p$-cocartesian morphism $x\to x'$ lying over $f$. 
\item If $\cal C$ is an $\infty$-category, we let $\cal C^{\simeq}$ denote
the subcategory of $\cal C$ spanned by the equivalences of $\cal C$.
\end{itemize}

\section*{Acknowledgment}

The author appreciates the anonymous referee for carefully examining
this paper, improving the exposition, and suggesting that the proof
of Proposition \ref{prop:sec1_main}, which was very complicated in
the first draft of this paper, should admit an easier proof.

\section{\label{sec:HA}Results Freely Used from Higher Algebra}

While the goal of this paper is to expand on the details of the proof
of FTOK given in Lurie's book \textit{Higher Algebra} \cite{HA},
we (of course) need to rely on constructions and results of loc. cit.
to achieve this goal. This puts us in a difficult position, where
it can sometimes be unclear which results are deemed to have complete
proofs and which are not. To overcome this, we list the results that
we freely use (and which the author believes has complete proofs)
from \cite{HA} below.

\subsection{Results on Cocartesian Fibrations}
\begin{lem}
\cite[Lemma 2.2.4.11]{HA}\label{lem:2.2.4.11} Let $p:\cal C\to\cal D$
be a cocartesian fibration of $\infty$-categories. Suppose there
is a full subcategory $\cal X\subset\cal C$ with the following properties:

\begin{enumerate}[label=(\roman*)]

\item For each object $D\in\cal D$, the inclusion $\cal X_{D}\subset\cal C_{D}$
admits a left adjoint $L_{D}:\cal E_{D}\to\cal X_{D}$.

\item For each morphism $f:D\to D'$ in $\cal D$, the associated
functor $f_{!}:\cal C_{D}\to\cal C_{D'}$ carries $L_{D}$-equivalences
(i.e., its image under $L_{D}$ is an equivalence) to $L_{D'}$-equivalences.

\end{enumerate}

Let $q=p\vert_{\cal X}:\cal X\to\cal D$ denote the restriction of
$p$. Then the following holds:
\begin{enumerate}
\item The functor $q$ is a cocartesian fibration.
\item Let $g:D\to D'$ be a morphism in $\cal D$, and let $f:X\to Y$ be
a morphism in $\cal X$ lifting $g$. Then $f$ is $q$-cocartesian
if and only if the map $g_{!}C\to C'$ is an $L_{D'}$-equivalence,
where $g_{!}:\cal C_{D}\to\cal C_{D'}$ is the functor induced by
$g$.
\end{enumerate}
\end{lem}

\begin{lem}
\cite[Lemma 2.4.4.6]{HA}\label{lem:2.4.4.6} Let $\cal C$ be an
$\infty$-category, let $0<i<n$ be integers, and let $p:\cal C\to\Delta^{n}$
be a functor. Suppose that for each object $X\in\cal C_{i-1}$, there
is a $p$-cocartesian morphism $f:X\to Y$ with $Y\in\cal C_{i}$.
Then the inclusion
\[
\cal C\times_{\Delta^{n}}\Lambda_{i}^{n}\hookrightarrow\cal C
\]
is a weak categorical equivalence.
\end{lem}

\subsection{Results on Operadic Colimit Diagrams}
\begin{prop}
\cite[Proposition 3.1.1.7]{HA}\label{prop:3.1.1.7} Let $q:\cal C^{\t}\to\cal O^{\t}$
be a fibration of $\infty$-operads, and let $\overline{p}:K^{\rcone}\to\cal C_{\act}^{\t}$
be a diagram. The following conditions are equivalent:
\begin{enumerate}
\item The map $\overline{p}$ is a weak operadic $q$-colimit diagram.
\item For every $n>0$ and every diagram
\[\begin{tikzcd}
	{K\star \partial\Delta^n} & {\mathcal{C}^\otimes _{\mathrm{act}}} \\
	{K\star \Delta^n} & {\mathcal{O}^\otimes _{\mathrm{act}}}
	\arrow["{f_0}", from=1-1, to=1-2]
	\arrow[from=1-1, to=2-1]
	\arrow[from=1-2, to=2-2]
	\arrow["f", dashed, from=2-1, to=1-2]
	\arrow["{\overline{f}}"', from=2-1, to=2-2]
\end{tikzcd}\]such that the restriction of $f_{0}$ to $K\star\{0\}$ coincides
with $\overline{p}$ and $f_{0}\pr n\in\cal C$, there exists a dotted
arrow $f$ rendering the diagram commutative.
\end{enumerate}
\end{prop}

\begin{prop}
\cite[Proposition 3.1.1.8]{HA}\label{prop:3.1.1.8} Let $q:\cal C^{\t}\to\cal O^{\t}$
be a fibration of $\infty$-operads, and let $\{\overline{p}_{i}:K_{i}^{\rcone}\to\cal C_{\act}^{\t}\}_{i\in I}$
a finite collection of operadic $q$-colimit diagrams. Set $K=\prod_{i\in I}K_{i}$,
and let $\overline{p}$ denote the composition
\[
K^{\rcone}\to\prod_{i\in I}K_{i}^{\rcone}\to\prod_{i\in I}\cal C_{\act}^{\t}\xrightarrow{\bigoplus_{I}}\cal C_{\act}^{\t}.
\]
Then $\overline{p}$ is an operadic $q$-colimit diagram.
\end{prop}

\subsection{Miscellany}
\begin{lem}
\cite[Lemma 3.1.2.5]{HA}\label{lem:3.1.2.5} Let $\cal C$ be an
$\infty$-category and $\cal C^{0}\subset\cal C$ a full subcategory.
Let $\sigma:\Delta^{n}\to\cal C$ be a nondegenerate simplex such
that $\sigma\pr i\not\in\cal C^{0}$ for each $0\leq i\leq n$. Consider
the following simplicial sets:
\begin{enumerate}
\item The simplicial subset $K\subset\cal C$ consisting of those simplices
$\tau:\Delta^{k}\star\Delta^{l}\to\cal C$, where $k,l\geq-1$, $\tau\vert\Delta^{k}$
factors through $\cal C^{0}$, and $\tau\vert\Delta^{l}$ factors
through $\sigma$.
\item The simplicial subset $K_{0}\subset\cal C$ consisting of those simplices
$\tau:\Delta^{k}\star\Delta^{l}\to\cal C$, where $k,l\geq-1$, $\tau\vert\Delta^{k}$
factors through $\cal C^{0}$, and $\tau\vert\Delta^{l}$ factors
through $\sigma\vert\partial\Delta^{n}$.
\end{enumerate}
Then the map $K_{0}\amalg_{\cal C_{/\sigma}^{0}\star\partial\Delta^{n}}\pr{\cal C_{/\sigma}^{0}\star\Delta^{n}}\to K$
is a trivial cofibration in the Joyal model structure.
\end{lem}

\begin{rem}
In the final paragraph of the proof of Lemma \ref{lem:3.1.2.5} in
\cite{HA}, a certain diagram of simplicial sets (denoted by $\{\cal C_{/\sigma_{J}}^{0}\}_{J\in\cal I_{0}^{\op}}$)
is claimed to be projectively cofibrant in the projective model structure,
without further explanations. One way to prove the projective cofibrancy
is to observe that the indexing category of the diagram (denoted by
$\cal I_{0}^{\op}$) has a natural structure of a direct category,
and then to use a criterion for projectively cofibrant diagrams in
diagram categories indexed by direct categories \cite[Theorem 5.1.3]{Hovey2007}.
\end{rem}

\section{\label{sec:families}A Result on Families of \texorpdfstring{$\infty$}{Infinity}-Operads}

Let $\cal C$ be an $\infty$-category. Recall that a $\cal C$\textbf{-family
of $\infty$-operads} \cite[Definition 2.3.2.10]{HA} is a categorical
fibration $p:\cal M^{\t}\to\cal C\times N\pr{\Fin_{\ast}}$ satisfying
the following conditions:
\begin{itemize}
\item [(a)]For each object $M\in\cal M^{\t}$ with image $\pr{C,\inp m}\in\cal C\times N\pr{\Fin_{\ast}}$
and for each inert map $\alpha:\inp m\to\inp n$ in $N\pr{\Fin_{\ast}}$,
the morphism $\pr{\id_{C},\alpha}$ admits a $p$-cocartesian lift.
\item [(b)]Let $M\in\cal M^{\t}$ be an object with image $\pr{C,\inp m}\in\cal C\times N\pr{\Fin_{\ast}}$,
where $m\geq1$. Suppose we are given, for each $1\leq i\leq m$,
a $p$-cocartesian lift $f_{i}:M\to M_{i}$ over $\pr{\id_{C},\rho^{i}}:\pr{C,\inp n}\to\pr{C,\inp 1}$.
Then the morphisms $f_{i}$ form a $p$-limit cone. Moreover, every
object in $\cal M_{\inp 0}^{\t}$ is $p$-terminal.
\item [(c)]Let $n\geq1$ and $C\in\cal C$. Given objects $M_{1},\dots,M_{n}\in\cal M_{C}$,
there is an object $M\in\cal M^{\t}$ lying over $\pr{C,\inp n}$
which admits $p$-cocartesian morphisms $\{M\to M_{i}\}_{1\leq i\leq n}$
over $\{\pr{\id_{C},\rho^{i}}:\pr{C,\inp n}\to\pr{C,\inp 1}\}_{1\leq i\leq n}$.
Moreover, $\cal M_{\pr{C,\inp 0}}^{\t}$ is non-empty.
\end{itemize}
The goal of this section is to prove the following equivalent formulation
of families of $\infty$-operads:
\begin{prop}
\label{prop:sec1_main}Let $\cal C$ be an $\infty$-category, and
let $p:\cal M^{\t}\to\cal C\times N\pr{\Fin_{\ast}}$ be a categorical
fibration satisfying conditions (a) and (b) above. Then the following
conditions are equivalent:
\begin{itemize}
\item [(c-i)]The map $p$ satisfies condition (c).
\item [(c-ii)]For each $1\leq i\leq n$, let $\rho_{!}^{i}:\cal M_{\inp n}^{\t}\to\cal M_{\inp 1}^{\t}$
denote the functor over $\cal C$ induced by the morphism $\rho^{i}$.
Then for each $n\geq1$, the functor
\[
\pr{\rho_{!}^{i}}_{1\leq i\leq n}:\cal M_{\inp n}^{\t}\to\cal M\times_{\cal C}\cdots\times_{\cal C}\cal M
\]
is an equivalence of $\infty$-categories. Moreover, the functor $\cal M_{\inp 0}^{\t}\to\cal C$
is a trivial fibration.
\end{itemize}
\end{prop}

Here the functor $\rho_{!}^{i}:\cal M_{\inp n}^{\t}\to\cal M$ is
obtained in the following way: Since every inert morphism in $\cal C\times N\pr{\Fin_{\ast}}$
admits a $p$-cocartesian lift, it is possible to choose a cocartesian
natural transformation $\cal M_{\inp n}^{\t}\times\Delta^{1}\to\cal M^{\t}$
fitting into the commutative diagram 
\[\begin{tikzcd}
	{\mathcal{M}^\otimes _{\langle n\rangle}\times \{0\}} & {\mathcal{M}^\otimes} \\
	{\mathcal{M}^\otimes _{\langle n\rangle}\times \Delta^1} & {\mathcal{C}\times N(\mathsf{Fin}_\ast).}
	\arrow[from=1-1, to=1-2]
	\arrow[from=1-1, to=2-1]
	\arrow[from=1-2, to=2-2]
	\arrow[dashed, from=2-1, to=1-2]
	\arrow["{p_2\times \rho^i}"', from=2-1, to=2-2]
\end{tikzcd}\]We understand that $\rho_{!}^{i}$ is the restriction of the filler,
so that it is a functor over $\cal C$ and its homotopy class over
$\cal C$ is well-defined.
\begin{rem}
Readers familiar with generalized $\infty$-operads will find that
Proposition \ref{prop:sec1_main} is more or less a rephrasing of
\cite[Proposition 2.3.2.11]{HA}. We include a proof nonetheless,
since we were not able to find the exact statement of Proposition
\ref{prop:sec1_main} in the literature.
\end{rem}

For the proof of Proposition \ref{prop:sec1_main}, we introduce a
bit of notation:
\begin{notation}
Given an inner fibration of simplicial sets $p:X\to S$ and a map
of simplicial sets $K\to S$, we let $\Fun_{S}^{\cc}\pr{K,X}\subset\Fun_{S}\pr{K,X}$
denote the full subcategory spanned by the maps $K\to X$ that carry
each edge in $K$ to a $p$-cocartesian morphisms.
\end{notation}

\begin{proof}
[Proof of Proposition \ref{prop:sec1_main}]Obviously (c-ii) implies
(c-i). Conversely, suppose that condition (c-i) is satisfied. We must
show that $p$ satisfies condition (c-ii). Since the categorical fibration
$\cal M_{\inp 0}^{\t}\to\cal C$ is fully faithful (by condition (b))
and essentially surjective (by condition (c)), it is a trivial fibration.
Therefore, it suffices to show that, for each $n\geq1$, the functor
\[
\pr{\rho_{!}^{i}}_{1\leq i\leq n}:\cal M_{\inp n}^{\t}\to\cal M\times_{\cal C}\cdots\times_{\cal C}\cal M
\]
is a categorical equivalence.

Set $\cal I=\{1,\dots,n\}\star\{\inp 0\}$. We define a functor $\cal I^{\lcone}\to N\pr{\Fin_{\ast}}$
by mapping the cone point $\infty\in\cal I^{\lcone}$ to $\inp n$,
the morphism $\infty\to i$ to $\rho^{i}:\inp n\to\inp 1$ for $1\leq i\leq n$,
and the morphism $i\to\inp 0$ to the unique morphism $\inp 1\to\inp 0$.
By condition (a), the functor $p:\cal M^{\t}\to N\pr{\Fin_{\ast}}$
admits cocartesian morphisms over inert maps. Therefore, the functor
\[
\pi:\Fun_{N\pr{\Fin_{\ast}}}^{\cc}\pr{\cal I^{\lcone},\cal M^{\t}}\to\Fun_{N\pr{\Fin_{\ast}}}\pr{\{\infty\},\cal M^{\t}}\cong\cal M_{\inp n}^{\t}
\]
is a trivial fibration (\cite[Proposition 4.3.2.15]{HTT}). Choose
a section $s$ of $\pi$. By the definition of $\pr{\rho^{i}}_{1\leq i\leq n}$,
the diagram
\[\begin{tikzcd}[column sep =small, scale cd =.9]
	{\mathcal{M}^\otimes _{\langle n \rangle}} & {\operatorname{Fun}^{\mathrm{cc}}_{N(\mathsf{Fin}_\ast)}(\mathcal{I}^{\triangleleft},\mathcal{M}^\otimes )} & {\operatorname{Fun}^{\mathrm{cc}}_{N(\mathsf{Fin}_\ast)}(\mathcal{I},\mathcal{M}^\otimes )} \\
	&& {\operatorname{Fun}^{\mathrm{cc}}_{N(\mathsf{Fin}_\ast)}(\Delta^1,\mathcal{M}^\otimes )^n\times _{(\mathcal{M}^\otimes _{\langle 0\rangle})^n}\mathcal{M}^\otimes _{\langle 0\rangle}} \\
	{\mathcal{M}^n\times _{\mathcal{C}^n}\mathcal{C}} & {\mathcal{M}^n\times _{\mathcal{C}^n}\mathcal{C}^n\times _{\mathcal{C}^n }\mathcal{C}} & {\mathcal{M}^n\times _{\operatorname{Fun}(\{0\},\mathcal{C})^n}\operatorname{Fun}^\simeq(\Delta^1,\mathcal{C})^n\times _{\operatorname{Fun}(\{1\},\mathcal{C})^n}\mathcal{C}}
	\arrow["s", from=1-1, to=1-2]
	\arrow["\simeq"', from=1-1, to=1-2]
	\arrow["{(\rho^i)_{1\leq i\leq n}}"', from=1-1, to=3-1]
	\arrow["\gamma", from=1-2, to=1-3]
	\arrow["\cong", from=1-3, to=2-3]
	\arrow["\beta", from=2-3, to=3-3]
	\arrow["\simeq"', from=2-3, to=3-3]
	\arrow["\cong"', from=3-1, to=3-2]
	\arrow["\alpha"', from=3-2, to=3-3]
	\arrow["\simeq", from=3-2, to=3-3]
\end{tikzcd}\]commutes up to natural equivalence. Here $\Fun^{\simeq}\pr{\Delta^{1},\cal C}\subset\Fun\pr{\Delta^{1},\cal C}$
denotes the full subcategory spanned by the equivalences, the map
$\alpha$ is the categorical equivalence induced by the diagonal map
$\cal C\to\Fun^{\simeq}\pr{\Delta^{1},\cal C}$, and the map $\Delta^{1}\to N\pr{\Fin_{\ast}}$
corresponds to the morphism $\inp 1\to\inp 0$. We also observe that
the functor $\beta$ is a categorical equivalence. Indeed, the vertical
arrows of the square
\[\begin{tikzcd}
	{\operatorname{Fun}^{\mathrm{cc}}_{N(\mathsf{Fin}_\ast)}(\Delta^1,\mathcal{M}^\otimes )} & {\operatorname{Fun}^{\simeq}(\Delta^1,\mathcal{C})} \\
	{\mathcal{M}} & {\operatorname{Fun}(\{0\},\mathcal{C})}
	\arrow[from=1-1, to=1-2]
	\arrow[from=1-1, to=2-1]
	\arrow[from=1-2, to=2-2]
	\arrow[from=2-1, to=2-2]
\end{tikzcd}\]are trivial fibrations, so the functor
\[
\Fun_{N\pr{\Fin_{\ast}}}^{\cc}\pr{\Delta^{1},\cal M^{\t}}\to\cal M\times_{\Fun\pr{\{0\},\cal C}}\Fun^{\simeq}\pr{\Delta^{1},\cal C}
\]
is a categorical equivalence. Since the functor $\cal M_{\inp 0}^{\t}\to\cal C$
is a categorical equivalence as observed above, we deduce that $\beta$
is a categorical equivalence. To complete the proof, therefore, it
will suffice to show that $\gamma$ is a categorical equivalence. 

Let $q:\cal M^{\t}\to N\pr{\Fin_{\ast}}$ denote the projection. By
\cite[Proposition 4.3.2.15]{HTT}, it will suffice to prove the following
assertions:
\begin{enumerate}
\item Every diagram $F\in\Fun_{N\pr{\Fin_{\ast}}}^{\cc}\pr{\cal I,\cal M^{\t}}$
can be extended to a diagram in $\Fun_{N\pr{\Fin_{\ast}}}^{\cc}\pr{\cal I^{\lcone},\cal M^{\t}}$.
\item Let $\overline{F}\in\Fun_{N\pr{\Fin_{\ast}}}\pr{\cal I^{\lcone},\cal M^{\t}}$
be a diagram whose restriction $F=\overline{F}\vert\cal I$ belongs
to $\Fun_{N\pr{\Fin_{\ast}}}^{\cc}\pr{\cal I,\cal M^{\t}}$. Then
$\overline{F}$ belongs to $\Fun_{N\pr{\Fin_{\ast}}}^{\cc}\pr{\cal I,\cal M^{\t}}$
if and only if it is a $q$-limit diagram.
\end{enumerate}

We start with assertion (1). Since $F$ carries each morphism to a
$q$-cocartesian morphism, we may assume without loss of generality
that the image of $F$ factors through the fiber $\cal M_{C}^{\t}$
for some $C\in\cal C$. Since $\cal M_{C}^{\t}$ is an $\infty$-operad,
we can extend $F$ to a diagram in $\Fun_{N\pr{\Fin_{\ast}}}^{\cc}\pr{\cal I^{\lcone},\cal M^{\t}}$,
as desired.

Next, for (2), suppose first that $\overline{F}$ belongs to $\Fun_{N\pr{\Fin_{\ast}}}^{\cc}\pr{\cal I,\cal M^{\t}}$.
We must show that $\overline{F}$ is a $q$-limit diagram. Since the
composite $\cal I^{\lcone}\xrightarrow{\overline{F}}\cal M^{\t}\to\cal C$
carries each morphism to an equivalence and $\cal I$ is weakly contractible,
it is a limit diagram (\cite[\href{https://kerodon.net/tag/03F3}{Tag 03F3}]{kerodon}).
It follows that $p\circ\overline{F}$ is a limit diagram relative
to the projection $\cal C\times N\pr{\Fin_{\ast}}\to N\pr{\Fin_{\ast}}$.
Thus, by \cite[Proposition 4.3.1.5]{HTT}, it suffices to show that
$\overline{F}$ is a $p$-limit diagram. By condition (b), $\overline{F}$
is a $p$-right Kan extension of $\overline{F}\vert\{\infty\}\star\pr{\{1,\dots,n\}}$,
which is a $p$-limit diagram (again by condition (b)). It follows
from \cite[Lemma 4.3.2.7]{HTT} that $\overline{F}$ is a $p$-limit
diagram, and the proof is complete.
\end{proof}

\section{\label{sec:env}Monoidal Envelopes of Families of \texorpdfstring{$\infty$}{Infinity}-Operads}

In \cite[Section 2.2.4]{HA}, Lurie introduces a universal procedure
to make an arbitrary $\infty$-operad into a symmetric monoidal $\infty$-category.
The resulting symmetric monoidal $\infty$-category is called the
\textbf{symmetric monoidal envelope} of the original $\infty$-operad.
In this section, we will generalize this construction to families
of $\infty$-operads. More precisely, we show that, given a family
of $\infty$-operads, we can take the symmetric monoidal envelope
of each fiber to obtain a family of symmetric monoidal $\infty$-categories.
We will also prove that the symmetric monoidal envelopes of families
of $\infty$-operads enjoy the expected universal property. 
\begin{rem}
In \cite{BHS22}, Barkan, Haugseng, and Steinebrunner develop a generalization
of symmetric monoidal envelopes that not only applies to $\infty$-operads
but also to a wide class of \textit{weak Segal fibrations} over \textit{algebraic
patterns} in the sense of Haugseng--Chu \cite{CH19}. There are various
similarities between their construction and our definition of symmetric
monoidal envelopes of families of $\infty$-operads (Definition \ref{def:env}),
but as far as the author can tell, our definition does not quite fit
into the framework of Barkan--Haugseng--Steinebrunner. The easiest
way to see this is that there is no guarantee that the symmetric monoidal
envelope of a family of $\infty$-operads in our sense is a cocartesian
fibration, while envelopes in \cite{BHS22} are automatically cocartesian
fibrations.
\end{rem}

We start with the definition of symmetric monoidal envelopes of families
of $\infty$-operads.
\begin{defn}
\label{def:env}Let $\cal C$ be an $\infty$-category, and let $p:\cal M^{\t}\to\cal C\times N\pr{\Fin_{\ast}}$
be a $\cal C$-family of $\infty$-operads. We define the\textbf{
symmetric monoidal envelope} $\Env\pr{\cal M}^{\t}$ of $\cal M^{\t}$
by the pullback
\[
\Env\pr{\cal M}^{\t}=\cal M^{\t}\times_{\Fun\pr{\{0\},N\pr{\Fin_{\ast}}}}\Fun^{\act}\pr{\Delta^{1},N\pr{\Fin_{\ast}}},
\]
where $\Fun^{\act}\pr{\Delta^{1},N\pr{\Fin_{\ast}}}\subset\Fun\pr{\Delta^{1},N\pr{\Fin_{\ast}}}$
denotes the full subcategory spanned by the active morphisms. We will
regard $\Env\pr{\cal M}^{\t}$ as equipped with the functor $\Env\pr{\cal M}^{\t}\to\cal C\times N\pr{\Fin_{\ast}}$
determined by the evaluation at the vertex $1\in\Delta^{1}$. We will
also regard $\Env\pr{\cal M}^{\t}$ as equipped with a (fully faithful)
functor $\cal M^{\t}\to\Env\pr{\cal M}^{\t}$ by using the diagonal
functor $N\pr{\Fin_{\ast}}\to\Fun^{\act}\pr{\Delta^{1},N\pr{\Fin_{\ast}}}$.
\end{defn}

\begin{rem}
Let $\cal C$ be an $\infty$-category, and let $p:\cal M^{\t}\to\cal C\times N\pr{\Fin_{\ast}}$
be a $\cal C$-family of $\infty$-operads. 
\begin{itemize}
\item The $\infty$-category $\Env\pr{\cal M}=\Env\pr{\cal M}_{\inp 1}^{\t}$
can be identified with the subcategory $\cal M_{\act}^{\t}\subset\cal M^{\t}$
of active morphisms. The intuition here is that an object $M\in\cal M_{\inp n}^{\t}$
is regarded as a ``formal tensor product'' of the objects $\rho_{!}^{i}\pr M\in\cal M$.
\item The objects of $\Env\pr{\cal M}^{\t}$ are the pairs $\pr{M,\alpha:p\pr M\to\inp k}$,
where $M\in\cal M^{\t}$ and $\alpha$ is an active morphism of $N\pr{\Fin_{\ast}}$.
\item Suppose that $\cal C=\Delta^{0}$, so that $p$ is just an $\infty$-operad.
Then $\Env\pr{\cal M}^{\t}$ is the symmetric monoidal envelope of
$\cal M^{\t}$ defined in \cite[Definition 2.2.4.1]{HTT}.
\item Suppose we are given an object $C\in\cal C$. There is an isomorphism
of simplicial sets
\[
\Env\pr{\cal M}_{C}^{\t}\cong\Env\pr{\cal M_{C}^{\t}}
\]
compatible with the projection to $N\pr{\Fin_{\ast}}$. We can informally
summarize this situation as follows: The $\infty$-category $\Env\pr{\cal M}^{\t}$
is obtained from $\cal M^{\t}$ by taking the symmetric monoidal envelopes
of its fibers.
\end{itemize}
\end{rem}

The goal of this section is to prove the following proposition, which
establishes the basic properties of symmetric monoidal envelopes of
families of $\infty$-operads. To state it, we must introduce a bit
of terminology.
\begin{defn}
If $\cal C$ is an $\infty$-category, we will say that a $\cal C$-family
of $\infty$-operads is a\textbf{ $\cal C$-family of symmetric monoidal
$\infty$-categories} if it admits cocartesian morphisms over the
morphisms in $\cal C^{\simeq}\times N\pr{\Fin_{\ast}}$.
\end{defn}

\begin{prop}
\label{prop:env}Let $\cal C$ be an $\infty$-category, and let $p:\cal M^{\t}\to\cal C\times N\pr{\Fin_{\ast}}$
be a $\cal C$-family of $\infty$-operads. Then:
\begin{enumerate}
\item The functor $q:\Env\pr{\cal M}^{\t}\to\cal C\times N\pr{\Fin_{\ast}}$
is a $\cal C$-family of symmetric monoidal $\infty$-categories.
\item A morphism of $\Env\pr{\cal M}^{\t}$ lying over a morphism of $\cal C^{\simeq}\times N\pr{\Fin_{\ast}}$
is $q$-cocartesian if and only if its image in $\cal M^{\t}$ is
inert.
\item For any $\cal C$-family $\cal N^{\t}\to\cal C\times N\pr{\Fin_{\ast}}$
of symmetric monoidal $\infty$-categories, the functor 
\[
\Fun_{\cal C\times N\pr{\Fin_{\ast}}}^{\t}\pr{\Env\pr{\cal M}^{\t},\cal N^{\t}}\to\Fun_{\cal C\times N\pr{\Fin_{\ast}}}^{\Alg}\pr{\cal M^{\t},\cal N^{\t}}
\]
is a categorical equivalence, where:
\begin{itemize}
\item $\Fun_{\cal C\times N\pr{\Fin_{\ast}}}^{\t}\pr{\Env\pr{\cal M}^{\t},\cal N^{\t}}\subset\Fun_{\cal C\times N\pr{\Fin_{\ast}}}\pr{\Env\pr{\cal M}^{\t},\cal N^{\t}}$
denotes the full subcategory spanned by the functors such that, for
each $C\in\cal C$, the map $\Env\pr{\cal M_{C}}^{\t}\to\cal N_{C}^{\t}$
is a symmetric monoidal functor; and
\item $\Fun_{\cal C\times N\pr{\Fin_{\ast}}}^{\Alg}\pr{\cal M^{\t},\cal N^{\t}}\subset\Fun_{\cal C\times N\pr{\Fin_{\ast}}}\pr{\cal M^{\t},\cal N^{\t}}$
denotes full subcategory spanned by the functors such that, for each
$C\in\cal C$, the functor $\cal M_{C}^{\t}\to\cal N_{C}^{\t}$ is
a map of $\infty$-operads.
\end{itemize}
\end{enumerate}
\end{prop}

The remainder of this section is devoted to the proof of Proposition
\ref{prop:env}. The bulk of the proof will be in describing a procedure
to adjoin cocartesian morphisms to categorical fibrations (Proposition
\ref{prop:BHS22}), which we now discuss.
\begin{notation}
Let $p:\cal E\to\cal B$ and $p':\cal E'\to\cal B$ be inner fibrations
of $\infty$-categories, and let $\cal B_{0}\subset\cal B$ be a subcategory.
Suppose that $p$ and $p'$ admit cocartesian morphisms over the morphisms
in $\cal B_{0}$. We let $\Fun_{\cal B}^{\cal B_{0}\-\cc}\pr{\cal E,\cal E'}$
denote the full subcategory of $\Fun_{\cal B}\pr{\cal E,\cal E'}$
spanned by those functors $\cal E\to\cal E'$ that carry $p$-cocartesian
morphisms lying over morphisms in $\cal B_{0}$ to $p'$-cocartesian
morphisms.
\end{notation}

\begin{notation}
Let $\cal C$ be an $\infty$-category equipped with a factorization
system $\pr{\cal C_{L},\cal C_{R}}$. We let $\Fun^{R}\pr{\Delta^{1},\cal C}\subset\Fun\pr{\Delta^{1},\cal C}$
denote the full subcategory spanned by the morphisms in $\cal C_{R}$.
\end{notation}

\begin{prop}
\label{prop:BHS22}Let $\cal B$ be an $\infty$-category equipped
with a factorization system $\pr{\cal B_{L},\cal B_{R}}$, let $\cal C$
be an $\infty$-category, and let $p:\cal E\to\cal C\times\cal B$
be a categorical fibration that admits cocartesian morphisms over
the morphisms in $\cal C^{\simeq}\times\cal B_{L}$. Set $\cal X=\cal E\times_{\Fun\pr{\{0\},\cal B}}\Fun^{R}\pr{\Delta^{1},\cal B}$.
Then:
\begin{enumerate}
\item The functor $p':\cal X\to\cal C\times\Fun\pr{\{1\},\cal B}$ admits
cocartesian morphisms over the morphisms in $\cal C^{\simeq}\times\cal B$,
and a morphism of $\cal X$ lying over a morphism of $\cal C^{\simeq}\times\cal B$
is $p'$-cocartesian if and only if its image in $\cal E$ is a $p$-cocartesian
morphism lying over a morphism in $\cal C^{\simeq}\times\cal B_{L}$.
\item For every categorical fibration $q:\cal F\to\cal C\times\cal B$ that
admits cocartesian morphisms over the morphisms in $\cal C^{\simeq}\times\cal B$,
the functor $\cal E\to\cal X$ (induced by the diagonal functor $\cal B\to\Fun^{R}\pr{\Delta^{1},\cal B}$)
induces a categorical equivalence
\[
\Fun_{\cal C\times\cal B}^{\cal C^{\simeq}\times\cal B\-\cc}\pr{\cal X,\cal F}\xrightarrow{\simeq}\Fun_{\cal C\times\cal B}^{\cal C^{\simeq}\times\cal B_{L}\-\cc}\pr{\cal E,\cal F}.
\]
\end{enumerate}
\end{prop}

\begin{rem}
In the case where $\cal C=\Delta^{0}$, Proposition \ref{prop:BHS22}
gives a formula for the ``free cocartesian fibration'' on a categorical
fibration satisfying certain conditions. This formula for free cocartesian
fibrations has appeared previously in various places, such as in \cite[Proposition A.0.1]{AMGR24}
and \cite[Proposition 2.2.4]{BHS22}. In the very special case where
$\cal C=\Delta^{0}$ and $\cal B=N\pr{\Fin_{\ast}}$ and $p$ is a
fibration of $\infty$-operads, the essence of the proposition appears
in \cite[$\S$2.2.4]{HA}, where Lurie defined symmetric monoidal envelopes. 

The proof of Proposition \ref{prop:BHS22} we present below is a synthesis
of the arguments in \cite{AMGR24,BHS22,HA}. Unfortunately, the proofs
of Proposition \ref{prop:BHS22} (for the case where $\cal C=\Delta^{0}$)
presented in \cite{AMGR24} and \cite{BHS22} do not seem to stand
as they are.\footnote{The authors of these papers promised (via private communications)
that the problem will be fixed in future revisions.} (They claim that the inverse equivalence of the functor appearing
in part (2) is given by a precomposition by a certain functor $\cal X\to\cal E$,
but the latter functor is not compatible with the projection to $\cal B$.)
Because of this, we will refrain from referring to the results of
\cite{AMGR24, BHS22} in what follows, at the cost of being somewhat
repetitive.
\end{rem}

The proof of Proposition \ref{prop:BHS22} relies on a few lemmas.

\begin{lem}
\label{lem:fact}Let $\cal B$ be an $\infty$-category, and let $\pr{\cal B_{L},\cal B_{R}}$
be a factorization system on $\cal B$. Then:
\begin{enumerate}
\item The functor $q:\Fun^{R}\pr{\Delta^{1},\cal B}\to\Fun\pr{\{1\},\cal B}$
is a cocartesian fibration.
\item A morphism of $\Fun^{R}\pr{\Delta^{1},\cal B}$ is $q$-cocartesian
if and only if its image in $\Fun\pr{\{0\},\cal B}$ belongs to $\cal B_{L}$.
\end{enumerate}
\end{lem}

\begin{proof}
Recall that the functor $p:\Fun\pr{\Delta^{1},\cal B}\to\Fun\pr{\{1\},\cal B}$
is a cocartesian fibration, and the $p$-cocartesian morphsims are
the morphisms whose images in $\Fun\pr{\{0\},\cal B}$ are equivalences
\cite[\href{https://kerodon.net/tag/0478}{Tag 0478}]{kerodon}. We
will apply Lemma \ref{lem:2.2.4.11} to $p$ and the full subcategory
$\Fun^{R}\pr{\Delta^{1},\cal B}$. Unwinding the definitions, it will
suffice to prove the following:
\begin{itemize}
\item [($\ast$)]Suppose we are given a $2$-simplex $\Delta^{2}\to\cal B$
depicted as 
\[\begin{tikzcd}
	X && {X'} \\
	& B
	\arrow["l", from=1-1, to=1-3]
	\arrow["f"', from=1-1, to=2-2]
	\arrow["r", from=1-3, to=2-2]
\end{tikzcd}\]where $l\in\cal B_{L}$ and $r\in\cal B_{R}$. Then for each map $r':Y\to B$
in $\cal B_{R}$, the map
\[
\Map_{\cal B_{/B}}\pr{X',Y}\to\Map_{\cal B_{/B}}\pr{X,Y}
\]
is an isomorphism of homotopy types. 
\end{itemize}
But assertion ($\ast$) is just a paraphrase of the fact that $l$
is left orthogonal to $r'$ (\cite[Remark 5.2.8.3]{HTT}), so there
is nothing to prove.
\end{proof}
\begin{lem}
\label{lem:arr}Let $\cal C$ be an $\infty$-category, and let $f:X\to Y$
be a morphism of $\cal C$. Let $\Fun^{\simeq}\pr{\Delta^{1},\cal C}\subset\Fun\pr{\Delta^{1},\cal C}$
denote the full subcategory spanned by the equivalences of $\cal C$.
The $\infty$-category $\Fun^{\simeq}\pr{\Delta^{1},\cal C}\times_{\Fun\pr{\Delta^{1},\cal C}}\Fun\pr{\Delta^{1},\cal C}_{/f}$
has a final object, given by the morphism $\sigma:\id_{X}\to f$ depicted
as 
\[\begin{tikzcd}
	X & X \\
	X & Y
	\arrow[equal, from=1-1, to=1-2]
	\arrow[equal, from=1-1, to=2-1]
	\arrow["f", from=1-2, to=2-2]
	\arrow["f"', from=2-1, to=2-2]
\end{tikzcd}\]
\end{lem}

\begin{proof}
By \cite[\href{https://kerodon.net/tag/03XA}{Tag 03XA}]{kerodon},
it suffices to show that, for each equivalence $e:C\xrightarrow{\simeq}C'$
in $\cal C$, composition with $\sigma$ induces an isomorphism of
homotopy types
\[
\Map_{\Fun\pr{\Delta^{1},\cal C}}\pr{e,\id_{X}}\xrightarrow{\simeq}\Map_{\Fun\pr{\Delta^{1},\cal C}}\pr{e,f}.
\]
This is immediate from \cite[Proposition 4.3.2.15]{HTT}, which says
that the maps $\Map_{\Fun\pr{\Delta^{1},\cal C}}\pr{e,\id_{X}}\to\Map_{\cal C}\pr{C,X}$
and $\Map_{\Fun\pr{\Delta^{1},\cal C}}\pr{e,f}\to\Map_{\cal C}\pr{C,X}$
are isomorphisms of homotopy types.
\end{proof}
\begin{proof}
[Proof of Proposition \ref{prop:BHS22}]For part (1), factor $p'$
as
\[
\cal X\xrightarrow{\phi}\cal C\times\Fun^{R}\pr{\Delta^{1},\cal B}\xrightarrow{\psi}\cal C\times\Fun\pr{\{1\},\cal B}.
\]
By Lemma \ref{lem:fact}, the functor $\psi$ is a cocartesian fibration
whose cocartesian morphisms are the morphisms whose images in $\Fun\pr{\{0\},\cal B}$
belong to $\cal B_{L}$. Combining this observation with the hypothesis
on $p$, we deduce that $\phi$ (, which is a pullback of $p$) admits
cocartesian morphisms over $\psi$-cocartesian morphisms lying over
the morphisms in $\cal C^{\simeq}\times\cal B$. Hence $p'$ admits
cocartesian morphisms over the morphisms in $\cal C^{\simeq}\times\cal B$.
Moreover, this argument shows that a morphism of $\cal X$ lying over
$\cal C^{\simeq}\times\cal B$ is $p'$-cocartesian if and only if
its image in $\cal E$ is a $p$-cocartesian morphism lying over a
morphism of $\cal C^{\simeq}\times\cal B_{L}$. Thus we have proved
(1). 

For (2), let $\Fun^{\simeq}\pr{\Delta^{1},\cal B}\subset\Fun\pr{\Delta^{1},\cal B}$
denote the full subcategory spanned by the equivalences of $\cal B$,
and set $\widetilde{\cal E}=\cal E\times_{\Fun\pr{\{0\},\cal B}}\Fun^{\simeq}\pr{\Delta^{1},\cal B}$.
The functor $\cal E\to\widetilde{\cal E}$ is a categorical equivalence,
so it suffices to show that the inclusion $\widetilde{\cal E}\subset\cal X$
induces a categorical equivalence
\[
\Fun_{\cal C\times\cal B}^{\cal C^{\simeq}\times\cal B\-\cc}\pr{\cal X,\cal F}\xrightarrow{\simeq}\Fun_{\cal C\times\cal B}^{\cal C^{\simeq}\times\cal B_{L}\-\cc}\pr{\widetilde{\cal E},\cal F}.
\]
For this, by \cite[Proposition 4.3.2.15]{HTT}, it suffices to prove
the following pair of assertions:

\begin{enumerate}[label=(\alph*)]

\item Every functor $F_{0}\in\Fun_{\cal C\times\cal B}^{\cal C^{\simeq}\times\cal B_{L}\-\cc}\pr{\widetilde{\cal E},\cal F}$
admits a $q$-left Kan extension $F\in\Fun_{\cal C\times\cal B}\pr{\cal X,\cal F}$.

\item Let $F\in\Fun_{\cal C\times\cal B}\pr{\cal X,\cal F}$ be a
functor whose restriction $F_{0}=F\vert\widetilde{\cal E}$ belongs
to $\Fun_{\cal C\times\cal B}^{\cal C^{\simeq}\times\cal B_{L}\-\cc}\pr{\widetilde{\cal E},\cal F}$.
Then $F$ belongs to $\Fun_{\cal C\times\cal B}^{\cal C^{\simeq}\times\cal B\-\cc}\pr{\cal X,\cal F}$
if and only if it is a $q$-left Kan extension of $F_{0}$.

\end{enumerate}

We begin with (a). Let $\pr{E,p\pr E\to B}$ be an arbitrary object
of $\cal X$ with image $\pr{C,B}\in\cal C\times\cal B$. Using Lemma
\ref{lem:arr}, we deduce that the $\infty$-category $\widetilde{\cal E}\times_{\cal X}\cal X_{/\pr{E,p\pr E\to B}}$
admits a final object, depicted as 
\[\begin{tikzcd}
	E & {p(E)} & {p(E)} \\
	E & {p(E)} & B.
	\arrow[equal, from=1-1, to=2-1]
	\arrow[equal, from=1-2, to=1-3]
	\arrow[equal, from=1-2, to=2-2]
	\arrow["f", from=1-3, to=2-3]
	\arrow["f"', from=2-2, to=2-3]
\end{tikzcd}\]Thus, to prove (a), it suffices to show that for each object $E\in\cal E$
and each morphism $f:p\pr E\to B$ of $\cal B_{R}$, there is a $q$-cocartesian
morphism that lifts the morphism $\pr{\id_{C},f}$ in $\cal C\times\cal B$
and has source $F\pr{E,\id_{p\pr E}}$. This follows from the hypothesis
that $q$ admits cocartesian morphisms over the morphisms in $\cal C^{\simeq}\times\cal B$.

Next, we prove (b). By the argument of (a), it suffices to prove that
the following conditions are equivalent:

\begin{enumerate}[label=(b-\roman*)]

\item The functor $F$ carries $p'$-cocartesian morphisms over the
morphisms in $\cal C^{\simeq}\times\cal B$ to $q$-cocartesian morphisms.

\item For each object $E\in\cal E$ and each morphism $f:p\pr E\to B$
in $\cal B_{R}$, the morphism $F\pr{E,\id_{p\pr E}}\to F\pr{E,f}$
is $q$-cocartesian.

\end{enumerate}

The implication (b-i)$\implies$(b-ii) follows from (1). For the converse,
suppose that condition (b-ii) is satisfied. Consider an arbitrary
$p'$-cocartesian morphism $\alpha$ lying over a morphism in $\cal C^{\simeq}\times\cal B$.
We depicted $\alpha$ as
\[\begin{tikzcd}
	E & {p(E)} & B \\
	{E'} & {p(E')} & {B',}
	\arrow["{\widetilde{l}}"', from=1-1, to=2-1]
	\arrow["r", from=1-2, to=1-3]
	\arrow["l"', from=1-2, to=2-2]
	\arrow["z", from=1-3, to=2-3]
	\arrow["{r'}"', from=2-2, to=2-3]
\end{tikzcd}\]where $l$ belongs to $\cal B_{L}$, $r$ and $r'$ belongs to $\cal B_{R}$,
$\widetilde{l}$ is $p$-cocartesian, and the image of $\widetilde{l}$
in $\cal C$ is an equivalence. We wish to show that $F\pr{\alpha}$
is $p$-cocartesian. We consider the diagram $\Delta^{1}\times\Delta^{1}\to\cal X$
depicted as 
\[\begin{tikzcd}
	& E && {p(E)} && B \\
	E && {p(E)} && {p(E)} \\
	& {E'} && {p(E')} && {B'.} \\
	{E'} && {p(E')} && {p(E')}
	\arrow["{\widetilde{l}}"', from=1-2, to=3-2]
	\arrow["r", from=1-4, to=1-6]
	\arrow["l"{pos=0.7}, from=1-4, to=3-4]
	\arrow["z", from=1-6, to=3-6]
	\arrow[equal, from=2-1, to=1-2]
	\arrow["{\widetilde{l}}"', from=2-1, to=4-1]
	\arrow[equal, from=2-3, to=1-4]
	\arrow[equal, from=2-3, to=2-5]
	\arrow["l"'{pos=0.7}, from=2-3, to=4-3]
	\arrow["r", from=2-5, to=1-6]
	\arrow["l"{pos=0.7}, from=2-5, to=4-5]
	\arrow["{r'}"{pos=0.7}, from=3-4, to=3-6]
	\arrow[equal, from=4-1, to=3-2]
	\arrow[equal, from=4-3, to=3-4]
	\arrow[equal, from=4-3, to=4-5]
	\arrow["{r'}"', from=4-5, to=3-6]
\end{tikzcd}\]By hypothesis, the morphisms $F\pr{E,\id_{p\pr E}}\to F\pr{E,r}$
and $F\pr{E',\id_{p\pr{E'}}}\to F\pr{E',r'}$ are $q$-cocartesian.
Since $F_{0}$ preserves cocartesian morphisms over morphisms in $\cal C^{\simeq}\times\cal B_{L}$,
the morphism $F\pr{E,\id_{p\pr E}}\to F\pr{E',\id_{p\pr{E'}}}$ is
$q$-cocartesian. Hence the morphism $F\pr{\alpha}:F\pr{E,r}\to F\pr{E,r'}$
is also $q$-cocartesian (\cite[Proposition 2.4.1.7]{HTT}), and the
proof is complete.
\end{proof}
We now arrive at the proof of Proposition \ref{prop:env}:
\begin{proof}
[Proof of Proposition \ref{prop:env}]Assertions (2) and (3) follow
from Proposition \ref{prop:BHS22}. We will complete the proof by
showing (1). We must verify the following conditions:

\begin{enumerate}[label=(\alph*)]

\item The functor $q$ admits cocartesian morphisms over the morphisms
in $\cal C^{\simeq}\times\cal B$.

\item Let $E\in\Env\pr{\cal M}^{\t}$ be an object with image $\pr{C,\inp n}\in\cal C\times N\pr{\Fin_{\ast}}$,
and let $\{f_{i}:E\to E_{i}\}_{1\leq i\leq n}$ be inert maps of $\Env\pr{\cal M_{C}}^{\t}$
lying over the maps $\{\rho^{i}:\inp n\to\inp 1\}_{1\leq i\leq n}$.
Then the maps $\{f_{i}\}_{1\leq i\leq n}$ form a $q$-limit diagram.
(If $n=0$, we interpret this as saying that every object of $\Env\pr{\cal M}^{\t}$
is $q$-terminal.)

\item Given an integer $n\geq0$, and objects $C\in\cal C$ and $E_{1},\dots,E_{n}\in\Env\pr{\cal M}_{C}$,
there is an object $E\in\Env\pr{\cal M}_{\pr{C,\inp n}}^{\t}$ that
admits $q$-cocartesian morphisms $\{E\to E_{i}\}_{1\leq i\leq n}$.
(If $n=0$, we interpret this as saying that $\cal M_{\pr{C,\inp 0}}^{\t}$
is non-empty.)

\end{enumerate}

Condition (a) follows from Proposition \ref{prop:BHS22}. Condition
(c) follows from (2) and the fact that $\cal M^{\t}$ is a $\cal C$-family
of $\infty$-operads. We will complete the proof by showing (b). We
factor the map $q$ as
\[
\Env\pr{\cal M}^{\t}\xrightarrow{r}\cal C\times\Fun^{\act}\pr{\Delta^{1},N\pr{\Fin_{\ast}}}\xrightarrow{s}\cal C\times\Fun\pr{\{1\},N\pr{\Fin_{\ast}}}.
\]
We observe that the functor $\Fun^{\act}\pr{\Delta^{1},N\pr{\Fin_{\ast}}}\to\Fun\pr{\{1\},N\pr{\Fin_{\ast}}}$
is a symmetric monoidal $\infty$-category, and a morphism of $\Fun^{\act}\pr{\Delta^{1},N\pr{\Fin_{\ast}}}$
is cocartesian if and only if its image in $\Fun\pr{\{0\},N\pr{\Fin_{\ast}}}$
is inert. It follows that $s$ is a $\cal C$-family of $\infty$-operads
and is also a cocartesian fibration, and a morphism of $\cal C\times\Fun^{\act}\pr{\Delta^{1},N\pr{\Fin_{\ast}}}$
is $s$-cocartesian if and only if its image in $\Fun\pr{\{0\},N\pr{\Fin_{\ast}}}$
is inert. Hence the maps $\{r\pr{f_{i}}\}_{1\leq i\leq n}$ form an
$s$-limit diagram. To complete the proof, therefore, it suffices
to show that $\{f_{i}\}_{1\leq i\leq n}$ is an $r$-limit diagram.
For this, it suffices to show that the image of $\{f_{i}\}_{1\le i\leq n}$
in $\cal M^{\t}$ is a $p$-limit diagram. This follows from (2) and
the hypothesis that $p$ is a $\cal C$-familiy of $\infty$-operads.
The proof is now complete.
\end{proof}

\section{\label{sec:direct_sum}The Fiberwise Direct Sum Functor}

Let $\cal O^{\t}$ be an $\infty$-operad. Given objects $X_{1},\dots,X_{n}\in\cal O$,
we can find an object $X\in\cal O_{\inp n}^{\t}$ which admits a $p$-cocartesian
morphism $X\to X_{i}$ over $\rho^{i}:\inp n\to\inp 1$ for each $i$.
Such an object is denoted by $X_{1}\oplus\cdots\oplus X_{n}$ \cite[Remark 2.1.1.15]{HA}.
By inspection, the object $X_{1}\oplus\cdots\oplus X_{n}$ is equivalent
to the image of the object $\pr{X_{1},\dots,X_{n}}\in\pr{\cal O_{\act}^{\t}}^{n}$
under the tensor product of $\Env\pr{\cal O}=\cal O_{\act}^{\t}$.
So the operation $\oplus$ can be made functorial using monoidal envelopes.
In fact, we can do it fiberwise:
\begin{defn}
Let $\cal C$ be an $\infty$-category, let $\cal M^{\t}\to\cal C\times N\pr{\Fin_{\ast}}$
be a $\cal C$-family of $\infty$-operads, and let $n\geq1$ be a
positive integer. We define a functor
\[
\bigoplus_{i=1}^{n}:\pr{\cal M_{\act}^{\t}}^{n}\times_{\cal C^{n}}\cal C=\cal M_{\act}^{\t}\times_{\cal C}\cdots\times_{\cal C}\cal M_{\act}^{\t}\to\cal M_{\act}^{\t}
\]
over $\cal C$, well-defined up to natural equivalence over $\cal C$,
as follows: According to Proposition \ref{prop:sec1_main}, the functors
$\rho_{!}^{i}:\Env\pr{\cal M}_{\inp n}^{\t}\to\Env\pr{\cal M}_{\inp 1}^{\t}=\cal M_{\act}^{\t}$
induce a categorical equivalence
\[
\Env\pr{\cal M}_{\inp n}^{\t}\xrightarrow{\simeq}\pr{\cal M_{\act}^{\t}}^{n}\times_{\cal C^{n}}\cal C
\]
over $\cal C$. The functors $\Env\pr{\cal M}_{\inp n}^{\t}\to\cal C$
and $\pr{\cal M_{\act}^{\t}}^{n}\times_{\cal C^{n}}\cal C\to\cal C$
are categorical fibrations, so the above equivalence admits an inverse
equivalence
\[
\pr{\cal M_{\act}^{\t}}^{n}\times_{\cal C^{n}}\cal C\xrightarrow{\simeq}\Env\pr{\cal M}_{\inp n}^{\t}
\]
over $\cal C$. We define $\bigoplus_{i=1}^{n}$ to be the composite
\[
\pr{\cal M_{\act}^{\t}}^{n}\times_{\cal C^{n}}\cal C\xrightarrow{\simeq}\Env\pr{\cal M}_{\inp n}^{\t}\xrightarrow{\text{forget}}\cal M_{\act}^{\t}.
\]
\end{defn}

In this section, we will prove two important properties of the fiberwise
direct sum functor: Its universal property (Subsection \ref{subsec:Universal-Property-of})
and its interaction with slices (Subsection \ref{subsec:dir_sum_slice}).

\subsection{\label{subsec:Universal-Property-of}Universal Property of the Direct
Sum Functor }

In this subsection, we will characterize the direct sum functor by
a certain universal property (Corollary \ref{cor:oplus_p-limit}).
\begin{prop}
\label{prop:inverting_weak_equiv_of_pairs}Let $\cal M$ be a model
category. Suppose we are given a commutative diagram 
\[\begin{tikzcd}
	E & {E'} \\
	B & {B'}
	\arrow["f", from=1-1, to=1-2]
	\arrow["{p'}", two heads, from=1-2, to=2-2]
	\arrow["p"', two heads, from=1-1, to=2-1]
	\arrow["q"', two heads, from=2-1, to=2-2]
	\arrow["\simeq", from=2-1, to=2-2]
	\arrow["\simeq"', from=1-1, to=1-2]
\end{tikzcd}\]in $\cal M$. Assume the following:
\begin{enumerate}
\item The maps $p,p',q$ are fibrations.
\item The maps $f$ and $q$ are weak equivalences.
\item The object $E'$ is cofibrant. 
\item The map $q$ has a section $s:B'\to B$.
\end{enumerate}
Then there is a map $g:E'\to E$ rendering the diagram 
\[\begin{tikzcd}
	{E'} & E \\
	{B'} & B
	\arrow["{p'}", from=1-1, to=2-1]
	\arrow["g", from=1-1, to=1-2]
	\arrow["s"', from=2-1, to=2-2]
	\arrow["p", from=1-2, to=2-2]
\end{tikzcd}\]commutative, such that the composite $fg:E'\to E'$ is homotopic to
the identity in $\cal M_{/B'}$.
\end{prop}

\begin{proof}
Consider the diagram 
\[\begin{tikzcd}
	E & {E'} & E \\
	B & {B'} & B
	\arrow["p"', from=1-3, to=2-3]
	\arrow["s"', from=2-2, to=2-3]
	\arrow["{p'}"', from=1-2, to=2-2]
	\arrow["q"', from=2-1, to=2-2]
	\arrow["p"', from=1-1, to=2-1]
	\arrow["f", from=1-1, to=1-2]
	\arrow[dashed, from=1-2, to=1-3]
\end{tikzcd}\]in $\cal M_{/B'}$. Since $p$ is a fibration between fibrant objects
of $\cal M_{/B'}$ and $E'$ is cofibrant, the dual of \cite[Proposition A.2.3.1]{HTT}
shows that there is a map $g:E'\to E$ rendering the diagram commutative,
such that $[g][f]=[\id_{E}]$ in $\ho\pr{\cal M_{/B'}}$. Since $[f]$
is an isomorphism in $\ho\pr{\cal M_{/B'}}$, the uniqueness of inverses
implies that $[f][g]=[\id_{E'}]$ in $\ho\pr{\cal M_{/B'}}$. Since
$E'$ is fibrant-cofibrant in $\cal M_{/B'}$, we deduce that $fg$
is homotopic to the identity in $\cal M_{/B'}$.
\end{proof}
\begin{defn}
Let $n\geq1$. Given integers $m_{1},\dots,m_{n}\geq0$, we shall
identify the pointed set $\Vee_{i=1}^{n}\inp{m_{i}}$ with the set
$\inp{m_{1}+\cdots+m_{n}}$ via the pointed map
\[
\Vee_{i=1}^{n}\inp{m_{i}}\to\inp{m_{1}+\cdots+m_{n}}
\]
which maps $k\in\inp{m_{i}}^{\circ}$ to $\sum_{j<i}m_{j}+k$. For
each $1\leq i\leq n$, we define a natural transformation
\[
h_{i}:N\pr{\Fin_{\ast}}_{\act}^{n}\times\Delta^{1}\to N\pr{\Fin_{\ast}}
\]
by using the inert maps $\Vee_{i=1}^{n}\inp{m_{i}}\to\inp{m_{i}}$.
\end{defn}

\begin{prop}
\label{prop:oplus_p-limit}Let $n\ge1$, let $\cal C$ be an $\infty$-category,
and let $p:\cal M^{\t}\to\cal C\times N\pr{\Fin_{\ast}}$ be a $\cal C$-family
of $\infty$-operads. We can construct the functor $\bigoplus_{1\leq i\leq n}:\pr{\cal M_{\act}^{\t}}^{n}\times_{\cal C^{n}}\cal C\to\cal M^{\t}$
so that for each $1\leq i\leq n$, there is an inert natural transformation
$\bigoplus_{1\leq i\leq n}\to\opn{pr}_{i}$ rendering the diagram
\begin{equation}\label{d:prop:oplus_p-limit}
\begin{tikzcd}
	{((\mathcal{M}^\otimes_{\mathrm{act}})^n\times _{\mathcal{C}^n}\mathcal{C})\times \Delta ^1} & {\mathcal{M}^\otimes } \\
	{\mathcal{C}\times N(\mathsf{Fin}_\ast )^n_{\mathrm{act}}\times \Delta^1} & {\mathcal{C}\times N(\mathsf{Fin}_\ast)}
	\arrow["{\widetilde{h}_i}", from=1-1, to=1-2]
	\arrow["p", from=1-2, to=2-2]
	\arrow["{\operatorname{id}_{\mathcal{C}}\times h_i}"', from=2-1, to=2-2]
	\arrow[from=1-1, to=2-1]
\end{tikzcd}
\end{equation}commutative.
\end{prop}

\begin{proof}
We begin with the construction of the functor $\bigoplus_{1\leq i\leq n}$.
For each $1\leq i\leq n$, there is an inert natural transformation
$g_{i}:\Env\pr{N\pr{\Fin_{\ast}}}_{\inp n}^{\t}\times\Delta^{1}\to\Env\pr{N\pr{\Fin_{\ast}}}^{\t}$
from the inclusion to the functor 
\[
\pr{\alpha:\inp k\to\inp n}\mapsto\pr{\alpha^{-1}\pr i_{\ast}\to\inp 1},
\]
where $\alpha^{-1}\pr i_{\ast}$ denotes the unique object $\inp m\in\Fin_{\ast}$
which admits an order-preserving bijection $\alpha^{-1}\pr i\cong\inp m^{\circ}$.
This natural transformation covers $\rho^{i}:\inp n\to\inp 1$. Thus
we may construct the functor $\rho_{!}^{i}:\Env\pr{N\pr{\Fin_{\ast}}}_{\inp n}^{\t}\to N\pr{\Fin_{\ast}}_{\act}$
by $\rho_{!}^{i}\pr{\alpha:\inp k\to\inp n}=\alpha^{-1}\pr i_{\ast}$.
Since the functor $p$ is a categorical fibration, it induces a fibration
$\Env\pr{\cal M}^{\t}\to\cal C\times\Env\pr{N\pr{\Fin_{\ast}}}^{\t}$
of generalized $\infty$-operads. Therefore, we can find an inert
natural transformation $\Env\pr{\cal M}_{\inp n}^{\t}\times\Delta^{1}\to\Env\pr{\cal M}^{\t}$
rendering the diagram 
\[\begin{tikzcd}
	{\operatorname{Env}(\mathcal{M})^\otimes _{\langle n\rangle}\times \{0\}} && {\operatorname{Env}(\mathcal{M})^\otimes } \\
	{\operatorname{Env}(\mathcal{M})^\otimes _{\langle n\rangle}\times \Delta^1} & {\mathcal{C}\times \operatorname{Env}(N(\mathsf{Fin}_\ast))^\otimes _{\langle n\rangle}\times \Delta^1} & {\mathcal{C}\times \operatorname{Env}(N(\mathsf{Fin}_\ast))^\otimes }
	\arrow[from=1-3, to=2-3]
	\arrow[from=1-1, to=2-1]
	\arrow[dashed, from=2-1, to=1-3]
	\arrow[from=2-1, to=2-2]
	\arrow["{\operatorname{id}_{\mathcal{C}}\times g_i}"', from=2-2, to=2-3]
	\arrow[from=1-1, to=1-3]
\end{tikzcd}\]commutative. We use this inert natural transformation to define the
functor $\rho_{!}^{i}:\Env\pr{\cal M}_{\inp n}^{\t}\to\cal M_{\act}^{\t}$.
This will ensure that the diagram

\begin{equation}\label{d1}
\begin{tikzcd}
	{(\mathcal{M}^\otimes _{\mathrm{act}})^n\times _{\mathcal{C}^n}\mathcal{C}} && {\operatorname{Env}(\mathcal{M})^\otimes _{\langle n \rangle}} \\
	{\mathcal{C}\times (N(\mathsf{Fin}_\ast)_{\mathrm{act}})^n} && {\mathcal{C}\times \operatorname{Env}(N(\mathsf{Fin}_\ast))^\otimes _{\langle n \rangle}}
	\arrow[from=1-1, to=2-1]
	\arrow["{\operatorname{id}_{\mathcal{C}}\times(\rho^i_!)_{i=1}^n}", from=2-3, to=2-1]
	\arrow[from=1-3, to=2-3]
	\arrow["{(\rho^i_!)_{i=1}^n}"', from=1-3, to=1-1]
\end{tikzcd}
\end{equation}commutes. Now the functor $\pr{\rho_{!}^{i}}_{i=1}^{n}:\Env\pr{N\pr{\Fin_{\ast}}}_{\inp n}^{\t}\to\pr{\pr{N\pr{\Fin_{\ast}}}_{\act}}^{n}$
is a trivial fibration, and it has a section $\phi:\pr{N\pr{\Fin_{\ast}}_{\act}}^{n}\to\Env\pr{N\pr{\Fin_{\ast}}}_{\inp n}^{\t}$
given by 
\[
\pr{\inp{k_{i}}}_{i=1}^{n}\mapsto\pr{\Vee_{i=1}^{n}\inp{k_{i}}\to\Vee_{i=1}^{n}\inp 1=\inp n}.
\]
Applying Proposition \ref{prop:inverting_weak_equiv_of_pairs} to
the commutative diagram (\ref{d1}) and the section $\id_{\cal C}\times\phi$,
we can find a functor $\Phi:\pr{\cal M_{\act}^{\t}}^{n}\times_{\cal C^{n}}\cal C\to\Env\pr{\cal M}_{\inp n}^{\t}$
with the following properties:
\begin{itemize}
\item The diagram 
\[\begin{tikzcd}
	{(\mathcal{M}^\otimes _{\mathrm{act}})^n\times _{\mathcal{C}^n}\mathcal{C}} && {\operatorname{Env}(\mathcal{M})^\otimes _{\langle n \rangle}} \\
	{\mathcal{C}\times (N(\mathsf{Fin}_\ast)_{\mathrm{act}})^n} && {\mathcal{C}\times \operatorname{Env}(N(\mathsf{Fin}_\ast))^\otimes _{\langle n \rangle}}
	\arrow[from=1-1, to=2-1]
	\arrow["{\operatorname{id}_{\mathcal{C}}\times\phi}"', from=2-1, to=2-3]
	\arrow[from=1-3, to=2-3]
	\arrow["\Phi", from=1-1, to=1-3]
\end{tikzcd}\]is commutative. 
\item The composite $\pr{\rho_{!}^{i}}_{1\leq i\leq n}\circ\Phi$ is naturally
equivalent over $\cal C\times\pr{N\pr{\Fin_{\ast}}_{\act}}^{n}$ to
the identity functor. 
\end{itemize}
We will construct the functor $\bigoplus_{1\leq i\leq n}$ as the
composite
\[
\pr{\cal M_{\act}^{\t}}^{n}\times_{\cal C^{n}}\cal C\xrightarrow{\Phi}\Env\pr{\cal M}_{\inp n}^{\t}\xrightarrow{\text{forget}}\cal M^{\t}.
\]

To complete the proof, we must show that, for each $1\leq i\leq n$,
there is an inert natural transformation $\widetilde{h}_{i}:\pr{\pr{\cal M_{\act}^{\t}}^{n}\times_{\cal C^{n}}\cal C}\times\Delta^{1}\to\cal M^{\t}$
rendering the diagram (\ref{d:prop:oplus_p-limit}) commutative. Since
$\opn{pr}_{i}$ is naturally equivalent over $\cal C\times N\pr{\Fin_{\ast}}$
to the composite $\rho_{!}^{i}\Phi$, it suffices (by \cite[Proposition A.2.3.1]{HTT})
to find an inert natural transformation rendering the diagram 
\[\begin{tikzcd}
	{((\mathcal{M}^\otimes_{\mathrm{act}})^n\times _{\mathcal{C}^n}\mathcal{C})\times \partial\Delta^1} && {\mathcal{M}^\otimes } \\
	{((\mathcal{M}^\otimes_{\mathrm{act}})^n\times _{\mathcal{C}^n}\mathcal{C})\times \Delta^1} & {\mathcal{C}\times N(\mathsf{Fin}_\ast)_{\mathrm{act}}^n\times \Delta^1} & {\mathcal{C}\times N(\mathsf{Fin}_\ast)}
	\arrow["{(\bigoplus_{i=1}^n,\rho^i_!\circ \Phi )}", from=1-1, to=1-3]
	\arrow[from=1-1, to=2-1]
	\arrow["p", from=1-3, to=2-3]
	\arrow[dashed, from=2-1, to=1-3]
	\arrow[from=2-1, to=2-2]
	\arrow["{\operatorname{id}_{\mathcal{C}}\times h_i}"', from=2-2, to=2-3]
\end{tikzcd}\]commutative. But the outer rectangle is equal to that of the diagram
\[\begin{tikzcd}[column sep=small, scale cd=.8]
	{((\mathcal{M}^\otimes_{\mathrm{act}})^n\times _{\mathcal{C}^n}\mathcal{C})\times \partial\Delta^1} & {\mathcal{C}\times \operatorname{Env}(\mathcal{M})^\otimes _{\langle n \rangle}\times \partial\Delta^1} && {\mathcal{M}^\otimes } \\
	{((\mathcal{M}^\otimes_{\mathrm{act}})^n\times_{\mathcal{C}^n}\mathcal{C})\times  \Delta^1} & {\mathcal{C}\times \operatorname{Env}(\mathcal{M})^\otimes _{\langle n\rangle}\times \Delta^1} & {\mathcal{C}\times \operatorname{Env}(N(\mathsf{Fin}_\ast))^\otimes _{\langle n\rangle}\times \Delta^1} & {\mathcal{C}\times N(\mathsf{Fin}_\ast),}
	\arrow[from=1-1, to=2-1]
	\arrow["{\operatorname{id}_{\mathcal{C}}\times g_i}"', from=2-3, to=2-4]
	\arrow[from=2-2, to=2-3]
	\arrow["{\Phi\times \operatorname{id}}"', from=2-1, to=2-2]
	\arrow["{\Phi\times \operatorname{id}}", from=1-1, to=1-2]
	\arrow["{(\text{forget}, \rho^i_!)}", from=1-2, to=1-4]
	\arrow[from=1-2, to=2-2]
	\arrow[from=1-4, to=2-4]
\end{tikzcd}\]so the existence of the desired filler follows from the definition
of $\rho_{!}^{i}$.
\end{proof}
\begin{cor}
[Universal Property of the Direct Sum Functor]\label{cor:oplus_p-limit}Let
$\cal C$ be an $\infty$-category, and let $p:\cal M^{\t}\to\cal C\times N\pr{\Fin_{\ast}}$
be a $\cal C$-family of $\infty$-operads. Let $q:K\to\cal C$ be
an object of $\SS_{/\cal C}$, let $n\geq1$, and let $f,g_{1},\dots,g_{n}:K\to\cal M_{\act}^{\t}$
be maps over $\cal C$. Assume the following:
\begin{enumerate}
\item For each $1\leq i\leq n$, there is an inert natural transformation
$\alpha_{i}:f\to g_{i}$ over $\cal C$.
\item For each vertex $v$ in $K$, the maps $\alpha_{i}$ give rise to
a $p$-limit cone $\{f\pr v\to g_{i}\pr v\}_{1\leq i\le n}$ lying
over the maps $\{\inp{m\pr v}\to\inp{m_{i}\pr v}\}_{1\leq i\leq n}$
of $N\pr{\Fin_{\ast}}$ that induces a bijection $\inp{m\pr v}^{\circ}\cong\coprod_{i=1}^{n}\inp{m_{i}\pr v}^{\circ}$.
\end{enumerate}
Then $f$ is naturally equivalent over $\cal C$ to the composite
\[
G:K\xrightarrow{\pr{g_{i}}_{i=1}^{n}}\pr{\cal M_{\act}^{\t}}^{n}\times_{\cal C^{n}}\cal C\xrightarrow{\bigoplus_{1\leq i\leq n}}\cal M_{\act}^{\t}.
\]
\end{cor}

\begin{proof}
Let $\eta:K\times\Delta^{1}\to N\pr{\Fin_{\ast}}$ denote the natural
equivalence which satisfies the following conditions: 
\begin{itemize}
\item The restriction $\eta\vert K\times\{0\}$ is given by $K\xrightarrow{\pr{p_{2}g_{i}}_{i}}N\pr{\Fin_{\ast}}_{\act}^{n}\xrightarrow{\Vee_{i=1}^{n}}N\pr{\Fin_{\ast}}$.
\item The restriction $\eta\vert K\times\{1\}$ is given by $p_{2}f$.
\item For each vertex $v\in K$ and $1\leq j\leq n$, the composite
\[
\Vee_{i=1}^{n}\inp{m_{i}\pr v}\xrightarrow{\eta}\inp{m\pr v}\xrightarrow{p_{2}\alpha_{j}}\inp{m_{j}\pr v}
\]
is the identity map on the $j$th summand and the null map on the
remaining summands.
\end{itemize}
Using the fact that $p$ is a categorical fibration, we can find a
functor $f':K\to\cal M^{\t}$ and a natural equivalence $f'\xrightarrow{\simeq}f$
which lifts the natural equivalence $\pr{q\circ\opn{pr}_{1},\eta}:K\times\Delta^{1}\to\cal C\times N\pr{\Fin_{\ast}}$.
Replacing $f$ by $f'$ if necessary, we may assume that the diagram
\[\begin{tikzcd}
	{K\times \Delta^1} & {\mathcal{M}^\otimes } \\
	{(\mathcal{C}\times (N(\mathsf{Fin}_\ast)_{\mathrm{act}})^n)\times \Delta^1} & {\mathcal{C}\times N(\mathsf{Fin}_\ast)}
	\arrow["{\alpha_i}", from=1-1, to=1-2]
	\arrow["p", from=1-2, to=2-2]
	\arrow["{\operatorname{id}_{\mathcal{C}}\times h_i}"', from=2-1, to=2-2]
	\arrow["{(q,(p_2g_j)_j)\times \operatorname{id}}"', from=1-1, to=2-1]
\end{tikzcd}\]commutes for each $1\leq i\leq n$. Since relative limits in functor
categories can be formed objectwise \cite[\href{https://kerodon.net/tag/02XK}{Tag 02XK}]{kerodon},
the morphisms $\alpha_{i}:f\to g_{i}$ in $\Fun\pr{K,\cal M^{\t}}$
form a $\Fun\pr{K,p}$-limit cone. Moreover, since relative limits
are stable under pullbacks \cite[Proposition 4.3.1.5]{HTT}, we deduce
that the morphisms $\{\alpha_{i}\}_{1\leq i\leq n}$ of $\Fun_{\cal C}\pr{K,\cal M^{\t}}$
form a $\pi$-limit cone. 

Now using Proposition \ref{prop:oplus_p-limit}, we can construct
the functor $\bigoplus_{1\leq i\leq n}:\pr{\cal M_{\act}^{\t}}^{n}\times_{\cal C^{n}}\cal C\to\cal M^{\t}$
so that there is an inert natural transformation $\widetilde{h}_{i}:\bigoplus_{1\leq i\leq n}\to\opn{pr}_{i}$
which lifts the composite $\pr{\cal M_{\act}^{\t}}^{n}\times_{\cal C^{n}}\cal C\to\cal C\times N\pr{\Fin_{\ast}}_{\act}^{n}\times\Delta^{1}\xrightarrow{\id_{\cal C}\times h_{i}}\cal C\times N\pr{\Fin_{\ast}}$.
Again, the induced natural transformations $\{G\to g_{i}\}_{1\leq i\leq n}$
form a $\Fun_{\cal C}\pr{K,p}$-limit cone covering the same cone
as $\{\alpha_{i}\}_{1\leq i\leq n}$. Thus we obtain the desired equivalence
$G\xrightarrow{\simeq}f$ in $\Fun_{\cal C}\pr{K,\cal M^{\t}}$, and
the proof is complete.
\end{proof}

\subsection{\label{subsec:dir_sum_slice}Direct Sum and Slice}

Let $\cal O^{\t}$ be an $\infty$-operad and let $Y\in\cal O^{\t}$
be an object. If $Y$ lies over an object $\inp n\in N\pr{\Fin_{\ast}}$
with $n\geq2$, then we can find objects $Y_{i}\in\cal O$ and an
equivalence $Y\simeq\bigoplus_{i=1}^{n}Y_{i}$. Given objects $X_{1},\dots,X_{n}\in\cal O^{\t}$
and active maps $f_{i}:X_{i}\to Y_{i}$, their direct sum $\bigoplus_{i=1}^{n}f_{i}:\bigoplus_{i=1}^{n}X_{i}\to\bigoplus_{i=1}^{n}Y_{i}\simeq Y$
is an active morphism with codomain $Y$. Conversely, given an active
morphism $f:X\to Y$ in $\cal O^{\t}$, we can write $f\simeq\bigoplus_{i=1}^{n}f_{i}$,
where $f_{i}$ is the active map obtained by factoring the composite
$X\to Y\to Y_{i}$ into an inert map followed by an active map. The
following proposition, which is the only result of this subsection,
asserts that this ``direct sum decomposition'' of morphisms is an
equivalence on the level of $\infty$-categories: 
\begin{prop}
\label{prop:directsum_slice_equiv}Let $\cal C$ be an $\infty$-category,
let $C\in\cal C$ be an object, and let $\cal M^{\t}\to\cal C\times N\pr{\Fin_{\ast}}$
be a $\cal C$-family of $\infty$-operads. For any integer $n\geq1$
and any objects $M_{1},\dots,M_{n}\in\cal M_{C}$, the direct sum
functor induces an equivalence of $\infty$-categories
\[
\pr{\pr{\cal M_{\act}^{\t}}^{n}\times_{\cal C^{n}}\cal C}_{/\pr{M_{1},\dots,M_{n}}}\xrightarrow{\simeq}\pr{\cal M_{\act}^{\t}}_{/\bigoplus_{i=1}^{n}M_{i}}.
\]
\end{prop}

\begin{proof}
Recall that the direct sum functor $\bigoplus_{i=1}^{n}$ is obtained
as the composite 
\[
\pr{\cal M_{\act}^{\t}}^{n}\times_{\cal C^{n}}\cal C\xrightarrow[\Phi]{\simeq}\Env\pr{\cal M}_{\inp n}^{\t}\xrightarrow{\text{forget}}\cal M_{\act}^{\t},
\]
where the map $\Phi$ is the inverse equivalence over $\cal C$ of
the functor $\pr{\rho_{!}^{i}}_{1\leq i\le n}:\Env\pr{\cal M}_{\inp n}^{\t}\xrightarrow{\simeq}\pr{\cal M_{\act}^{\t}}^{n}\times_{\cal C^{n}}\cal C$
. The functor $\Phi$ maps the object $\pr{M_{1},\dots,M_{n}}$ to
the object $\pr{\bigoplus_{i=1}^{n}M_{i},\alpha:\inp n\to\inp n}$,
where $\alpha$ is a bijection. It will therefore suffice to prove
that the functor 
\[
\theta:\pr{\Env\pr{\cal M}_{\inp n}^{\t}}_{/\pr{\bigoplus_{i=1}^{n}M_{i},\alpha}}\to\pr{\cal M_{\act}^{\t}}_{/\bigoplus_{i=1}^{n}M_{i}}
\]
is an equivalence of $\infty$-categories. Set $M=\bigoplus_{i=1}^{n}M_{i}$.
By definition, we have
\[
\pr{\Env\pr{\cal M}_{\inp n}^{\t}}_{/\pr{M,\alpha}}=\pr{\Env\pr{N\pr{\Fin_{\ast}}}_{\inp n}^{\t}}_{/\alpha}\times_{N\pr{\Fin_{\ast}}_{/\inp n}}\cal M_{/M}^{\t}.
\]
Since
\[
\pr{\Env\pr{N\pr{\Fin_{\ast}}}_{\inp n}^{\t}}_{/\alpha}\to\pr{N\pr{\Fin_{\ast}}_{\act}}_{/\inp n}
\]
is an \textit{isomorphism} of simplicial sets, so is $\theta$. In
particular, the map $\theta$ is a categorical equivalence, as claimed.
\end{proof}

\section{\label{sec:FTOK}The Fundamental Theorem of Operadic Kan Extensions}

The goal of this section is to reproduce Lurie's proof of the fundamental
theorem of operadic Kan extensions \cite[Theorem 3.1.2.3]{HA} and
indicate which part of the proof deserves further explanations. For
this, we start by recalling the statement of the theorem in Subsection
\ref{subsec:Recollection}. After that, we will give a transcription
of Lurie's proof. For the purpose of exposition, we will extract some
parts of his arguments as independent results in Subsection \ref{subsec:prelim},
and then reproduce Lurie's argument in Subsection \ref{subsec:Lurie's_proof}.
On the course of the retelling, we will indicate the nontrivial parts
of the arguments that was left out by Lurie, and state them as lemmas.
These lemmas will be proved in Section \ref{sec:leftover}, using
results from previous sections.

\subsection{\label{subsec:Recollection}Recollection}

In this subsection, we recall the definition of operadic Kan extensions
and the statement of the fundamental theorem of operadic Kan extensions. 
\begin{defn}
\cite[Definition 3.1.2.2]{HA} Let $\cal M^{\t}\to N\pr{\Fin_{\ast}}\times\Delta^{1}$
be a $\Delta^{1}$-family of $\infty$-operads and let $q:\cal C^{\t}\to\cal O^{\t}$
be a fibration of $\infty$-operads. Set $\cal A^{\t}=\cal M^{\t}\times_{\Delta^{1}}\{0\}$
and $\cal B^{\t}=\cal M^{\t}\times_{\Delta^{1}}\{1\}$. A map $F:\cal M^{\t}\to\cal C^{\t}$
is called an \textbf{operadic $q$-left Kan extension of $F\vert\cal A^{\t}$
}if the following condition is satisfied for every object $B\in\cal B$:
\begin{itemize}
\item [($\ast$)]The composite map
\[
\pr{\pr{\cal M_{\act}^{\t}}_{/B}\times_{\Delta^{1}}\{0\}}^{\rcone}\to\pr{\cal M_{/B}^{\t}}^{\rcone}\to\cal M^{\t}\xrightarrow{\overline{F}}\cal C^{\t}
\]
is an operadic $q$-colimit diagram.
\end{itemize}
\end{defn}

In fact, Lurie requires condition ($\ast$) to hold for every object
$B\in\cal B^{\t}$ in \cite[Definition 3.1.2.2]{HA}. This is not
a problem, because of the following proposition, which seems to be
implicit in \cite{HA}.
\begin{prop}
\label{prop:operadic_Kan_equiv}Let $p:\cal M^{\t}\to N\pr{\Fin_{\ast}}\times\Delta^{1}$
be a correspondence from an $\infty$-operad $\cal A^{\t}$ to another
$\infty$-operad $\cal B^{\t}$. Let $q:\cal C^{\t}\to\cal O^{\t}$
be a fibration of $\infty$-operads and let $F:\cal M^{\t}\to\cal C^{\t}$
be a map of generalized $\infty$-operads. Suppose that $F$ is an
operadic $q$-left Kan extension of \textbf{$F\vert\cal A^{\t}$}.
Then for every object $B\in\cal B^{\t},$ the map 
\[
\theta:\pr{\pr{\cal M_{\act}^{\t}}_{/B}\times_{\Delta^{1}}\{0\}}^{\rcone}\to\pr{\cal M_{/B}^{\t}}^{\rcone}\to\cal M^{\t}\xrightarrow{F}\cal C^{\t}
\]
is an operadic $q$-colimit diagram.
\end{prop}

\begin{proof}
Let $\inp n$ denote the image of $B$ in $N\pr{\Fin_{\ast}}$. We
consider two cases, depending on whether $n$ is equal to $0$ or
not.

Suppose first that $n=0$. Then we have $\pr{\cal M_{\act}^{\t}}_{/B}=\pr{\cal M_{\inp 0}^{\t}}_{/B}$,
since there is no active map $\inp k\to\inp 0$ with $k\geq1$. Since
the functor $\cal M_{\inp 0}^{\t}\to\Delta^{1}$ is a trivial fibration,
so is the functor $\pr{\cal M_{\act}^{\t}}_{/B}\times_{\Delta^{1}}\{0\}\to\Delta_{/1}^{1}\times_{\Delta^{1}}\{0\}\cong\Delta^{0}$.
It will therefore suffice to show that $F$ carries a morphism of
the form $A\to B$ in $\cal M^{\t}$ with $A\in\cal M_{\inp 0}^{\t}\times_{\Delta^{1}}\{0\}$
to an equivalence in $\cal C^{\t}$. This is clear, since $\cal C_{\inp 0}^{\t}$
is a contractible Kan complex.

Next, suppose that $n\geq1$. For each $1\leq i\leq n$, choose an
inert map $B\to B_{i}$ in $\cal B^{\t}$ over $\rho^{i}:\inp n\to\inp 1$,
and choose direct sum functors $\bigoplus_{i=1}^{n}$ for $\cal M^{\t}$
and $\cal C^{\t}$. Replacing the functor $\bigoplus_{i=1}^{n}$ by
a functor naturally equivalent one, we may assume that $B=\bigoplus_{i=1}^{n}B_{i}$.
According to Proposition \ref{prop:directsum_slice_equiv}, the direct
sum functor induces an equivalence of $\infty$-categories
\[
\phi:\pr{\pr{\cal M_{\act}^{\t}}^{n}\times_{\pr{\Delta^{1}}^{n}}\Delta^{1}}_{/\pr{B_{1},\dots,B_{n}}}\xrightarrow{\simeq}\pr{\cal M_{\act}^{\t}}_{/B}.
\]
There is also an isomorphism of simplicial sets
\[
\psi:\prod_{i=1}^{n}\pr{\pr{\cal M_{\act}^{\t}}_{/B_{i}}\times_{\Delta^{1}}\{0\}}\cong\pr{\pr{\cal M_{\act}^{\t}}^{n}\times_{\pr{\Delta^{1}}^{n}}\Delta^{1}}_{/\pr{B_{1},\dots,B_{n}}}\times_{\Delta^{1}}\{0\}.
\]
It will therefore suffice to show that the composite $\theta\circ\phi^{\rcone}\circ\psi^{\rcone}$
is an operadic $q$-colimit diagram. For this, we consider the diagram
\[\begin{tikzcd}[column sep = small, scale cd =.8]
	{(\prod_{i=1}^n((\mathcal{M}^\otimes _{\mathrm{act}})_{/B_i}\times _{\Delta^1}\{0\}))^\triangleright} & {(((\mathcal{M}^\otimes _{\mathrm{act}})^n\times_{(\Delta^1)^n}\Delta^1)_{/(B_1,\dots ,B_n)}\times_{\Delta^1}\{0\})^\triangleright} & {((\mathcal{M}^\otimes _{\mathrm{act}})_{/B}\times _{\Delta^1}\{0\})^\triangleright} \\
	{\prod_{i=1}^n((\mathcal{M}^\otimes _{\mathrm{act}})_{/B_i}\times _{\Delta^1}\{0\})^\triangleright} & {(\mathcal{M}^\otimes _{\mathrm{act}})^n\times_{(\Delta^1)^n}\Delta^1} & {\mathcal{M}^\otimes _{\mathrm{act}}} \\
	{(\mathcal{M}^\otimes _{\mathrm{act}})^n} & {(\mathcal{C}^\otimes _{\mathrm{act}})^n} & {\mathcal{C}^\otimes _{\mathrm{act}}}
	\arrow["\cong", from=1-1, to=1-2]
	\arrow["{\psi^\triangleright}"', from=1-1, to=1-2]
	\arrow[from=1-1, to=2-1]
	\arrow["\simeq", from=1-2, to=1-3]
	\arrow["{\phi^\triangleright}"', from=1-2, to=1-3]
	\arrow[from=1-2, to=2-2]
	\arrow[from=1-3, to=2-3]
	\arrow[from=2-1, to=3-1]
	\arrow["{\bigoplus_{i=1}^n}"', from=2-2, to=2-3]
	\arrow[hook', from=2-2, to=3-1]
	\arrow["G", from=2-2, to=3-2]
	\arrow["F", from=2-3, to=3-3]
	\arrow["{\prod_{i=1}^nF}"', from=3-1, to=3-2]
	\arrow["{\bigoplus_{i=1}^n}"', from=3-2, to=3-3]
\end{tikzcd}\]Here the map $G$ is the restriction of the map $\prod_{i=1}^{n}F$.
The left column and the top right square are commutative. The bottom
right square commutes up to natural equivalence by Proposition \ref{prop:oplus_p-limit}
and Corollary \ref{cor:oplus_p-limit}. Hence the diagram $\theta\circ\phi^{\rcone}\circ\psi^{\rcone}$
is naturally equivalent to the composite
\[
\theta':\pr{\prod_{i=1}^{n}\pr{\cal M_{\act}^{\t}}_{/B_{i}}\times_{\Delta^{1}}\{0\}}^{\rcone}\to\prod_{i=1}^{n}\pr{\pr{\cal M_{\act}^{\t}}_{/B_{i}}\times_{\Delta^{1}}\{0\}}^{\rcone}\to\pr{\cal C_{\act}^{\t}}^{n}\xrightarrow{\bigoplus_{i=1}^{n}}\cal C_{\act}^{\t}.
\]
Since the diagram $\theta'$ is an operadic $q$-colimit diagram by
Proposition \ref{prop:3.1.1.8}, we are done.
\end{proof}
We can now state the fundamental theorem of operadic Kan extensions.
\begin{thm}
\cite[Theorem 3.1.2.3]{HA}\label{thm:3.1.2.3} Let $n\geq1$, let
$p:\cal M^{\t}\to N\pr{\Fin_{\ast}}\times\Delta^{n}$ be a $\Delta^{n}$-family
of generalized $\infty$-operads, and let $q:\cal C^{\t}\to\cal O^{\t}$
be a fibration of $\infty$-operads. Consider a commutative diagram
\[\begin{tikzcd}
	{\mathcal{M}^\otimes \times _{\Delta^n}\Lambda ^n_0} & {\mathcal{C}^\otimes } \\
	{\mathcal{M}^\otimes } & {\mathcal{O}^\otimes}
	\arrow[hook, from=1-1, to=2-1]
	\arrow["{f_0}", from=1-1, to=1-2]
	\arrow["q", from=1-2, to=2-2]
	\arrow["g"', from=2-1, to=2-2]
	\arrow[dashed, from=2-1, to=1-2]
\end{tikzcd}\]of simplicial sets. Suppose that for each vertex $i$ in $\Lambda_{0}^{n}$
and each vertex $j$ in $\Delta^{n}$, the induced maps $\cal M^{\t}\times_{\Delta^{n}}\{i\}\to\cal C^{\t}$
and $\cal M^{\t}\times_{\Delta^{n}}\{j\}\to\cal O^{\t}$ are morphisms
of $\infty$-operads.
\begin{itemize}
\item [(A)]If $n=1$, the following conditions are equivalent:
\begin{itemize}
\item [(a)]There is a dashed filler which is an operadic $q$-left Kan
extension of $f_{0}$.
\item [(b)]For each object $B\in\cal M\times_{\Delta^{1}}\{1\}$, the diagram
\[
\{0\}\times_{\Delta^{1}}\pr{\pr{\cal M_{\act}^{\t}}_{/B}}\to\{0\}\times_{\Delta^{1}}\cal M^{\t}\xrightarrow{f_{0}}\cal C^{\t}
\]
admits an operadic $q$-colimit cone which lifts the map
\[
\pr{\{0\}\times_{\Delta^{1}}\pr{\pr{\cal M_{\act}^{\t}}_{/B}}}^{\rcone}\to\pr{\cal M_{/B}^{\t}}^{\rcone}\to\cal M^{\t}\xrightarrow{g}\cal O^{\t}.
\]
\end{itemize}
\item [(B)]If $n>1$ and the restriction $f_{0}\vert\cal M^{\t}\times_{\Delta^{n}}\Delta^{\{0,1\}}$
is an operadic $q$-left Kan extension of $f_{0}\vert\cal M^{\t}\times_{\Delta^{n}}\{0\}$,
then there is a dashed arrow rendering the diagram commutative.
\end{itemize}
\end{thm}

The rest of this section is devoted to reproducing Lurie's proof of
the above theorem.

\subsection{\label{subsec:prelim}Preliminary Results}

In this subsection, we collect some results that will facilitate the
proof of Theorem \ref{prop:operadic_Kan_equiv}. These results appeared
as parts of Lurie's proof of Theorem \ref{thm:3.1.2.3}, but for the
purpose of exposition, we will state them independently. We recommend
the reader to skip this subsection and only come back to it when the
need arises.
\begin{prop}
\label{prop:relative_terminal_obj_in_slice}Let $p:\cal C\to\cal D$
be an inner fibration of $\infty$-categories, $K$ a simplicial set,
and $\overline{f}:K^{\rcone}\to\cal C$ a diagram. Set $f=\overline{f}\vert K$,
and let $q$ denote the functor $\cal C_{f/}\to\cal D_{pf/}$. If
$\overline{f}$ maps the cone point to a $p$-terminal object, then
$\overline{f}\in\cal C_{f/}$ is $q$-terminal.
\end{prop}

\begin{proof}
We must show that the map 
\[
\pr{\cal C_{f/}}_{/\overline{f}}\to\cal C_{f/}\times_{\cal D_{pf/}}\pr{\cal D_{pf/}}_{/q\overline{f}}
\]
is a trivial fibration. Let $C=\overline{f}\pr{\infty}$. Given a
monomorphism $A\to B$ of simplicial sets, a lifting problem on the
left hand side corresponds under adjunction to a lifting problem on
the right hand side:
\[\begin{tikzcd}
	A & {(\mathcal{C}_{f/})_{/\overline{f}}} & {K\star A} & {\mathcal{C}_{/C}} \\
	B & {\mathcal{C}_{f/}\times _{\mathcal{D}_{pf/}}(\mathcal{D}_{pf/})_{/q\overline{f}}} & {K\star B} & {\mathcal{C}\times _{\mathcal{D}}\mathcal{D}_{/pC}.}
	\arrow[from=1-1, to=1-2]
	\arrow[from=1-2, to=2-2]
	\arrow[from=1-1, to=2-1]
	\arrow[from=2-1, to=2-2]
	\arrow[dashed, from=2-1, to=1-2]
	\arrow[from=1-3, to=2-3]
	\arrow[from=1-3, to=1-4]
	\arrow[from=1-4, to=2-4]
	\arrow[from=2-3, to=2-4]
	\arrow[dashed, from=2-3, to=1-4]
\end{tikzcd}\]The right hand lifting problem admits a solution because $C$ is $p$-terminal.
\end{proof}
\begin{cor}
\label{cor:Step2}Let $q:\cal C^{\t}\to\cal O^{\t}$ be a fibration
of $\infty$-operads, let $K$ be a simplicial set, and let $n\geq0$.
Consider a lifting problem 
\[\begin{tikzcd}
	{K\star \partial\Delta ^n} & {\mathcal{C}^\otimes } \\
	{K\star \Delta ^n} & {\mathcal{O}^\otimes.}
	\arrow[hook, from=1-1, to=2-1]
	\arrow["q", from=1-2, to=2-2]
	\arrow["f", from=1-1, to=1-2]
	\arrow["g"', from=2-1, to=2-2]
	\arrow[dashed, from=2-1, to=1-2]
\end{tikzcd}\]If $g$ maps the terminal vertex of $\Delta^{n}$ to an object in
$\cal O_{\inp 0}^{\t}$, then the lifting problem admits a solution.
\end{cor}

\begin{proof}
If $n=0$, find an object $C\in\cal C_{\inp 0}^{\t}$ which lies over
$g\pr 0$. Such an object exists because the functor $\cal C_{\inp 0}^{\t}\to\cal O_{\inp 0}^{\t}$
is a trivial fibration. The object $C$ is $q$-terminal, so the functor
\[
\cal C_{/C}^{\t}\to\cal O_{/q\pr C}^{\t}\times_{\cal O^{\t}}\cal C^{\t}
\]
is a trivial fibration. This implies the existence of the filler.

If $n>1$, then set $h=f\vert K$. It will suffice to solve the associated
lifting problem 
\[\begin{tikzcd}
	{\partial\Delta ^n} & {\mathcal{C}^\otimes _{h/}} \\
	{\Delta ^n} & {\mathcal{O}^\otimes_{qh/}.}
	\arrow[hook, from=1-1, to=2-1]
	\arrow["{q'}", from=1-2, to=2-2]
	\arrow[from=1-1, to=1-2]
	\arrow[from=2-1, to=2-2]
	\arrow[dashed, from=2-1, to=1-2]
\end{tikzcd}\]To solve this lifting problem, it suffices to show that the image
of the vertex $n\in\partial\Delta^{n}$ under the top horizontal arrow
is $q'$-terminal. This follows from Proposition \ref{prop:relative_terminal_obj_in_slice}.
\end{proof}
\begin{defn}
\label{def:head}Let $X$ be a simplicial set over $\Delta^{1}$.
Given a simplex $\Delta^{k}\to X$, we define the \textbf{head} of
$X$ to be the simplex $\Delta^{u^{-1}\pr 1}\to\Delta^{k}\to X$ and
the \textbf{tail} to be the simplex $\Delta^{u^{-1}\pr 0}\to\Delta^{k}\to X$,
where $u$ denotes the composite $\Delta^{k}\to X\to\Delta^{1}$.
\end{defn}

\begin{prop}
\label{prop:heads_and_tails}Let 
\[\begin{tikzcd}
	X && Y \\
	& {\Delta^1}
	\arrow["p", from=1-1, to=1-3]
	\arrow[from=1-3, to=2-2]
	\arrow[from=1-1, to=2-2]
\end{tikzcd}\]be a commutative diagram of simplicial sets. Let $n\geq0$, and let
$S\subset S'\subset Y_{1}=Y\times_{\Delta^{1}}\{1\}$ be simplicial
subsets satisfying the following conditions:
\begin{enumerate}
\item The simplicial set $S$ contains the $\pr{n-1}$-skeleton of $Y_{1}$.
\item The simplicial set $S'$ is generated by $S$ and a set $\Sigma$
of nondegenerate $n$-simplices of $Y_{1}$ that do not belong to
$S$.
\end{enumerate}
Let $X\pr S$ denote the simplicial subset of $X$ spanned by the
simplices whose head lies over $S$, and define $X\pr{S'}$ similarly.
Let $\{\sigma_{a}\}_{a\in A}$ be an enumeration of all nondegenerate
simplices of $X_{1}$ whose image in $Y$ is a degeneration of a simplex
in $\Sigma$. Choose a well-ordering of $A$ so that the dimension
of $\sigma_{a}$ is a non-decreasing function of $a\in A$. For each
$a\in A$, define simplicial sets $X\pr{S'}_{<a}$, $X\pr{S'}_{\leq a}$,
$K_{a}$, and $K_{0,a}$ as follows:
\begin{itemize}
\item $X\pr{S'}_{<a}\subset X$ is the simplicial subset generated by $X\pr S$
and those simplices of $X$ whose head factors through $\sigma_{b}$
for some $b<a$.
\item $X\pr{S'}_{\leq a}\subset X$ is the simplicial subset generated by
$X\pr S$ and those simplices of $X$ whose head factors through $\sigma_{b}$
for some $b\leq a$.
\item $K_{a}\subset X$ is the simplicial subset consisting of the simplices
whose head factors through $\sigma_{a}$.
\item $K_{0,a}\subset X$ is the simplicial subset consisting of the simplices
whose head factors through $\partial\sigma_{a}=\sigma_{a}\vert\partial\Delta^{\dim\sigma_{a}}$.
\end{itemize}
Then the following holds:
\begin{enumerate}
\item $X\pr{S'}=X\pr S\cup\bigcup_{a\in A}X\pr{S'}_{\leq a}$.
\item For each $a\in A$, we have $K_{0,a}\subset X\pr{S'}_{<a}$, and the
square 
\begin{equation}\label{d:h_and_t}
\begin{tikzcd}
	{K_{0,a}} & {X(S')_{<a}} \\
	{K_a} & {X(S')_{\leq a}}
	\arrow[from=1-1, to=2-1]
	\arrow[from=1-1, to=1-2]
	\arrow[from=1-2, to=2-2]
	\arrow[from=2-1, to=2-2]
\end{tikzcd}
\end{equation}of simplicial sets is cocartesian.
\end{enumerate}
\end{prop}

\begin{proof}
We start with (1). The containment $X\pr S\cup\bigcup_{a\in A}X\pr{S'}_{\leq a}\subset X\pr{S'}$
holds trivially. For the reverse inclusion, let $x$ be an arbitrary
simplex of $X\pr{S'}$. We must show that $x$ belongs to $X\pr S\cup\bigcup_{a\in A}X\pr{S'}_{\leq a}$.
Let $\eta\pr x:\Delta^{m}\to X_{1}$ denote the head of $x$. If $p\eta\pr x$
belongs to $S$, then $x$ belongs to $X\pr S$, and we are done.
If $\eta\pr x$ does not belong to $S$, then $p\eta\pr x$ factors
through a simplex $\sigma$ in $\Sigma$. If $p\eta\pr x$ factors
through the boundary of $\sigma$, then $p\eta\pr x$ belongs to the
$\pr{n-1}$-skeleton of $Y_{1}$ and hence $\eta\pr x$ belongs to
$S$, a contradiction. Therefore, $p\eta\pr x$ is a degeneration
of $\sigma$. Now $\eta\pr x$ is a degeneration of some nondegenerate
simplex $\eta'\pr x$ of $X_{1}$. Then $p\eta\pr x$ is a degeneration
of $p\eta'\pr x$, so $p\eta'\pr x$ is a degeneration of $\sigma$.
Hence $\eta'\pr x=\sigma_{a}$ for some $a\in A$, so that $x\in X\pr{S'}_{\leq a}$.
This proves (1).

We next prove (2). We will write $K_{0}=K_{0,a}$ and $K=K_{a}$.
First we show that $X\pr{S'}_{<a}$ contains $K_{0}$. Suppose we
are given a simplex $x$ of $X$ whose head $\eta\pr x:\Delta^{m}\to X_{1}$
factors through $\partial\sigma_{a}$. We wish to show that $x$ belongs
to $X\pr{S'}_{<a}$. By construction, there is a commutative diagram
\[\begin{tikzcd}
	{\partial\Delta^{\dim\sigma_a}} & {\Delta^{\dim\sigma_a}} & {\Delta^n} \\
	{\Delta^m} & {X_1} & {Y_1,}
	\arrow[hook, from=1-1, to=1-2]
	\arrow["s", from=1-2, to=1-3]
	\arrow["{\sigma_a}"', from=1-2, to=2-2]
	\arrow["\sigma", from=1-3, to=2-3]
	\arrow[from=2-1, to=1-1]
	\arrow["{\eta(x)}"', from=2-1, to=2-2]
	\arrow["p"', from=2-2, to=2-3]
\end{tikzcd}\]where $s$ is surjective on vertices, and where $\sigma\in\Sigma$.
If the map $\Delta^{m}\to\Delta^{n}$ is not surjective on vertices,
then $p\eta\pr x$ factors through the $\pr{n-1}$-skeleton of $Y_{1}$
and hence $x$ belongs to $X\pr S$. If the map $\Delta^{m}\to\Delta^{n}$
is surjective on vertices, then we write $\eta\pr x$ as a degeneration
of a nondegenerate simplex $\eta'\pr x$ of $X_{1}$. Since $\sigma$
is nondegenerate, $p\eta'\pr x$ is a degeneration of $\sigma$. Thus
$\eta'\pr x=\sigma_{b}$ for some $b\in A$. Since $\eta\pr x$ factors
through the boundary of $\sigma_{a}$, the dimension of $\sigma_{b}$
is strictly smaller than that of $\sigma_{a}$. Thus $b<a$. Hence
$x$ belongs to $X\pr{S'}_{<a}$, as desired.

Next, we show that the square (\ref{d:h_and_t}) is cocartesian. By
definition, $X\pr{S'}_{\leq a}$ is the union of $K$ and $X\pr{S'}_{<a}$.
Therefore, it suffices to show that $K_{0}$ is the intersection of
$K$ and $X\pr{S'}_{<a}$. So let $x$ be a simplex of $K$. We must
show that, if $x$ does not belong to $K_{0}$, then $x$ does not
belong to $X\pr{S'}_{<a}$ either. Assume, to the contrary, that $x$
belongs to $X\pr{S'}_{<a}$. Let $\eta\pr x:\Delta^{m}\to X_{1}$
be the head of $x$. Since $x$ does not belong to $K_{0}$, the simplex
$\eta\pr x$ is a degeneration of $\sigma_{a}$. Thus $p\pr{\eta\pr x}$
is a degeneration of some simplex $\sigma\in\Sigma$. In particular,
$p\pr{\eta\pr x}$ does not belong to $S$, so $\eta\pr x$ does not
belong to $X\pr S$. Since $x$ belongs to $X\pr{S'}_{<a}$ by hypothesis,
this means that $\eta\pr x$ factors through some $\sigma_{b}$ for
some $b<a$. If $\eta\pr x$ factors through $\partial\sigma_{b}$,
then we would have $\dim\sigma_{a}<\dim\sigma_{b}$, a contradiction.
So $\eta\pr x$ is a degeneration of $\sigma_{b}$; but then $a=b$,
a contradiction. Thus $x$ does not belong to $X\pr{S'}_{<a}$, as
claimed.
\end{proof}

\subsection{\label{subsec:Lurie's_proof}Lurie's Proof of Theorem \ref{thm:3.1.2.3}}

We now reproduce Lurie's proof of Theorem \ref{thm:3.1.2.3}. Along
the way, we will see that some parts of the proof merit further justifications.
We will explicitly state these parts as lemmas, and we will give proofs
to them in Section \ref{sec:leftover}. We stress that the argument
in this subsection is due to Lurie unless explicitly stated otherwise.

The proof proceeds by a simplex-by-simplex argument. For this, we
will classify simplices of $N\pr{\Fin_{\ast}}\times\Delta^{\{1,\dots,n\}}$
into five (somewhat artificial) groups, denoted by $G_{\pr 1},G_{\pr 2},G'_{\pr 2},G_{\pr 3},$and
$G'_{\pr 3}$.
\begin{defn}
Let $n\geq1$ and let $\alpha$ be a morphism in $N\pr{\Fin_{\ast}}\times\Delta^{\{1,\dots,n\}}$
with image $\alpha_{0}:\inp m\to\inp n$ in $N\pr{\Fin_{\ast}}$.
We say that $\alpha$ is:
\begin{enumerate}
\item \textbf{active} if $\alpha_{0}$ is active;
\item \textbf{strongly inert} if $\alpha_{0}$ is inert, the induced injection
$\inp n^{\circ}\to\inp m^{\circ}$ is order-preserving, and the image
of $\alpha$ in $\Delta^{\{1,\dots,n\}}$ is degenerate; and
\item \textbf{neutral} if it is neither active nor strongly inert.
\end{enumerate}
Note that active morphisms and strongly inert morphisms are closed
under composition. Also, every morphism in $N\pr{\Fin_{\ast}}\times\Delta^{\{1,\dots,n\}}$
can be factored uniquely as a composition of a strongly inert map
followed by an active map.

Let $\sigma$ be an $m$-simplex of $N\pr{\Fin_{\ast}}\times\Delta^{\{1,\dots,n\}}$
depicted as
\[
\pr{\inp{a_{0}},e_{0}}\xrightarrow{\alpha_{\sigma}\pr 1}\cdots\xrightarrow{\alpha_{\sigma}\pr m}\pr{\inp{a_{m}},e_{m}}.
\]
We will say that $\sigma$ is \textbf{closed} if $k_{m}=1$, and \textbf{open}
otherwise. We say that $\sigma$ is \textbf{complete}\footnote{Lurie uses te term ``new'' instead of ``complete.''}
if $\{e_{0},\dots,e_{m}\}=\{1,\dots,n\}$ and \textbf{incomplete}
otherwise. Note that every nondegenerate simplex of $N\pr{\Fin_{\ast}}\times\Delta^{\{1,\dots,n\}}$
is a face of a nondegenerate complete simplex. 

We partition the set of nondegenerate complete simplices of $N\pr{\Fin_{\ast}}\times\Delta^{\{1,\dots n\}}$
into five groups $G_{\pr 1},G_{\pr 2},G'_{\pr 2},G_{\pr 3},G_{\pr 3}'$
as follows: Let $\sigma$ be a nondegenerate complete $m$-simplex
of $N\pr{\Fin_{\ast}}\times\Delta^{\{1,\dots,n\}}$. Write $\alpha_{\sigma}\pr i=\sigma\vert\Delta^{\{i-1,i\}}$.
Let $0\le k\leq m$ be the minimal integer such that $\alpha_{\sigma}\pr i$
is strongly inert for every $i>k$, and let $0\leq j\leq k$ be the
minimal integer such that $\alpha_{\sigma}\pr i$ is active for every
$j<i\leq k$.
\begin{itemize}
\item If $j=0$, $k=m$, and $\sigma$ is closed, then $\sigma$ belongs
to $G_{\pr 1}$.
\item If $j=0$, $k<m$, and $\sigma$ is closed, then $\sigma$ belongs
to $G_{\pr 2}$.
\item If $j=0$ and $\sigma$ is open, then $\sigma$ belongs to $G'_{\pr 2}$.
\item If $j\geq1$ and $\alpha_{\sigma}\pr j$ is strongly inert, then $\sigma$
belongs to $G_{\pr 3}$.
\item If $j\geq1$ and $\alpha_{\sigma}\pr j$ is neutral, then $\sigma$
belongs to $G_{\pr 3}'$.
\end{itemize}
The situation is summarized in Table \ref{tab:G}, in which we used
the symbols $\mono$ and $\rightsquigarrow$ for strongly inert maps
and active maps, respectively.

We say that an $m$-simplex $\sigma\in G_{\pr 2}$ is an \textbf{associate}
of $\sigma'\in G_{\pr 2}'$ if $\sigma'=\sigma\vert\Delta^{\{0,\dots,m-1\}}$.
We also say that an $m$-simplex $\sigma\in G_{\pr 3}$ is an \textbf{associate}
of $\sigma'\in G'_{\pr 3}$ if $\sigma'=\sigma\partial_{j}\in G_{\pr 3}'$,
where $j$ is the integer defined above.
\end{defn}

\begin{table}[h]
\centering
\resizebox{\columnwidth}{!}{%
\begin{tabular}{|c|c|c|}
\hline 
Types & Pictures & Remarks\tabularnewline
\hline 
\hline 
$G_{\pr 1}$ & $\pr{\inp{a_{0}},e_{0}}\rightsquigarrow\cdots\rightsquigarrow\pr{\inp 1,e_{m}}$ & \tabularnewline
\hline 
$G_{\pr 2}$ & $\pr{\inp{a_{0}},e_{0}}\rightsquigarrow\cdots\rightsquigarrow\pr{\inp{a_{k}},e_{k}}\underset{\neq\emptyset}{\underbrace{\mono\cdots\mono}}\pr{\inp 1,e_{m}}$ & $k<m$.\tabularnewline
\hline 
$G'_{\pr 2}$ & $\pr{\inp{a_{0}},e_{0}}\rightsquigarrow\cdots\rightsquigarrow\pr{\inp{a_{k}},e_{k}}\mono\cdots\mono\pr{\inp{a_{m}},e_{m}}$ & $a_{m}\neq1,0\leq k\leq m$.\tabularnewline
\hline 
$G_{\pr 3}$ & $\cdots\mono\pr{\inp{a_{j}},e_{j}}\underset{\neq\emptyset}{\underbrace{\rightsquigarrow\cdots\rightsquigarrow}}\pr{\inp{a_{k}},e_{k}}\mono\cdots\mono\pr{\inp{a_{m}},e_{m}}$ & $1\leq j<k\leq m$.\tabularnewline
\hline 
$G'_{\pr 3}$ & $\cdots\xrightarrow{\text{neutral}}\pr{\inp{a_{j}},e_{j}}\rightsquigarrow\cdots\rightsquigarrow\pr{\inp{a_{k}},e_{k}}\mono\cdots\mono\pr{\inp{a_{m}},e_{m}}$ & $1\leq j\leq k\leq m$.\tabularnewline
\hline 
\end{tabular}%
}

\caption{Graphical Presentation of $G_{(0)}, G_{(1)}, G_{(2)}, G'_{(2)}, G_{(3)}, G'_{(3)}$}
\label{tab:G}
\end{table}

\begin{proof}
[Proof of Theorems \ref{thm:3.1.2.3}]We will regard $\cal M^{\t}$
and $N\pr{\Fin_{\ast}}\times\Delta^{n}$ as simplicial sets over $\Delta^{1}$
by means of the map $\Delta^{n}\to\Delta^{1}$ which maps the vertex
$0\in\Delta^{n}$ to the vertex $0\in\Delta^{1}$ and the remaining
vertices to the vertex $1\in\Delta^{1}$. Given a simplicial subset
$S\subset N\pr{\Fin_{\ast}}\times\Delta^{\{1,\dots,n\}}$, we let
$\cal M_{S}^{\t}\subset\cal M^{\t}$ denote the simplicial subset
consisting of those simplices whose head (Definition \ref{def:head})
lies over $S$.

The implication (a)$\implies$(b) for part (A) is obvious. Assume
therefore that condition (b) is satisfied if $n=1$. For each $m\geq0$,
let $F\pr m$ denote the simplicial subset of $N\pr{\Fin_{\ast}}\times\Delta^{\{1,\dots,n\}}$
generated by the nondegenerate simplices $\sigma$ satisfying one
of the following conditions:
\begin{itemize}
\item $\sigma$ is incomplete.
\item $\sigma$ has dimension less than $m$.
\item $\sigma$ has dimension $m$ and belongs to $G_{\pr 2}$ or $G_{\pr 3}$.
\end{itemize}
Observe that $\cal M_{F\pr 0}^{\t}=\cal M^{\t}\times_{\Delta^{n}}\Lambda_{0}^{n}$.
We will complete the proof by inductively constructing a map $f_{m}:\cal M_{F\pr m}^{\t}\to\cal C^{\t}$
which makes the diagram 
\[\begin{tikzcd}
	{\mathcal{M}^\otimes _{F(m-1)}} & {\mathcal{C}^\otimes} \\
	{\mathcal{M}^\otimes _{F(m)}} & {\mathcal{O}^\otimes }
	\arrow[hook, from=1-1, to=2-1]
	\arrow["{g\vert\mathcal{M}^\otimes _{F(m)}}"', from=2-1, to=2-2]
	\arrow["q", from=1-2, to=2-2]
	\arrow["{f_{m-1}}", from=1-1, to=1-2]
	\arrow["{f_{m}}"{description}, from=2-1, to=1-2]
\end{tikzcd}\]commutative, and such that $f_{1}$ has the following special properties
if $n=1$:
\begin{itemize}
\item [(i)]For each object $B\in\cal M\times_{\Delta^{1}}\{1\}$, the map
\[
\pr{\pr{\cal M_{\act}^{\t}}_{/B}\times_{\Delta^{1}}\{0\}}^{\rcone}\to\cal M_{F\pr 1}^{\t}\xrightarrow{f_{1}}\cal C^{\t}
\]
is an operadic $q$-colimit diagram.
\item [(ii)]For every inert morphism $e:M'\to M$ in $\cal M^{\t}\times_{\Delta^{1}}\{1\}$
such that $M\in\cal M$, the functor $f_{1}$ carries $e$ to an inert
morphism in $\cal C^{\t}$.
\end{itemize}
Fix $m>0$, and suppose that $f_{m-1}$ has been constructed. Observe
that $F\pr m$ is obtained from $F\pr{m-1}$ by adjoining the following
simplices:
\begin{itemize}
\item The $\pr{m-1}$-simplices in $G_{\pr 1}$.
\item The $\pr{m-1}$-simplices in $G'_{\pr 2}$ without associates.
\item The $m$-simplices in $G_{\pr 2}$ and $G_{\pr 3}$.
\end{itemize}
We define simplicial subsets $F'\pr m\subset F''\pr m\subset F\pr m$
as follows: $F'\pr m$ is generated by $F\pr{m-1}$ and the $\pr{m-1}$-simplices
in $G_{\pr 1}$; $F''\pr m$ is generated by $F'\pr m$ and the $\pr{m-1}$-simplices
of $G'_{\pr 2}$ without associates. Our strategy is to extend $f_{m-1}$
to $\cal M_{F'\pr m}^{\t}$, then to $\cal M_{F''\pr m}^{\t}$, and
then to $\cal M_{F\pr m}^{\t}$. 

\begin{enumerate}[label=(\textbf{Step \arabic*}),wide =0.5\parindent, listparindent=1.5em]

\item We will extend $f_{m-1}$ to a map $f'_{m}:\cal M_{F'\pr m}^{\t}\to\cal C^{\t}$
over $\cal O^{\t}$. To get an idea of how we will proceed, recall
that $F'\pr m$ is obtained from $F\pr{m-1}$ by adjoining the $\pr{m-1}$-simplices
in $G_{\pr 1}$, i.e., nondegenerate complete simplices consisting
of active maps and ending at $\pr{\inp 1,n}$. Roughly speaking, the
extension is possible because operadic colimits are defined in terms
of the active parts of the relevant $\infty$-operads, and $G_{\pr 1}$
stays in this realm.

The actual argument proceeds as follows: Let $\{\sigma_{a}\}_{a\in A}$
be the collection of nondegenerate simplices of $\cal M^{\t}\times_{\Delta^{n}}\Delta^{\{1,\dots,n\}}$
whose image in $N\pr{\Fin_{\ast}}\times\Delta^{\{1,\dots,n\}}$ is
a degeneration of some $\pr{m-1}$-simplex of $G_{\pr 1}$. Choose
a well-ordering on $A$ such that $\dim\sigma_{a}$ is non-decreasing
in $a\in A$. For each $a\in A$, let $\cal M_{<a}^{\t}$ denote the
simplicial subset spanned by $\cal M_{F\pr{m-1}}^{\t}$ and the simplices
of $\cal M^{\t}$ whose head factors through $\sigma_{b}$ for some
$b<a$. We define $\cal M_{\leq a}^{\t}$ similarly. According to
Proposition \ref{prop:heads_and_tails}, we have $\cal M_{F'\pr m}^{\t}=\cal M_{F\pr{m-1}}^{\t}\cup\bigcup_{a\in A}\cal M_{\leq a}^{\t}$,
so it suffices to extend $f_{m-1}$ to an $A$-sequence $f^{\leq a}:\cal M_{\leq a}^{\t}\to\cal C^{\t}$
over $\cal O^{\t}$. 

The construction is inductive. Let $a\in A$, and suppose that $f^{\leq b}$
has been constructed for $b<a$. These maps determine a map $f^{<a}:\cal M_{<a}^{\t}\to\cal C^{\t}$
extending $f_{m-1}$. Let $K_{0}\subset K\subset\cal M^{\t}$ denote
the simplicial subset consisting of the simplices of $\cal M^{\t}$
whose head factors through $\sigma_{a}\vert\partial\Delta^{\dim\sigma_{a}}$
and $\sigma_{a}$, respectively. According to Lemma \ref{lem:3.1.2.5},
the left hand square of the commutative diagram
\[\begin{tikzcd}
	{(\mathcal{M}^\otimes _{/\sigma_a}\times_{\Delta^n} \{0\})\star \partial\Delta ^{\dim\sigma _a}} & {K_0} & {\mathcal{M}^\otimes _{<a}} \\
	{(\mathcal{M}^\otimes _{/\sigma_a}\times_{\Delta^n} \{0\})\star \Delta ^{\dim\sigma _a}} & K & {\mathcal{M}^\otimes _{\leq a}}
	\arrow[from=1-1, to=1-2]
	\arrow[from=2-1, to=2-2]
	\arrow[from=1-2, to=2-2]
	\arrow[from=1-2, to=1-3]
	\arrow[from=1-3, to=2-3]
	\arrow[from=2-2, to=2-3]
	\arrow[from=1-1, to=2-1]
\end{tikzcd}\]is homotopy cocartesian. The right hand square is cocartesian by Proposition
\ref{prop:heads_and_tails}. It follows that the map
\[
\pr{\cal M_{/\sigma_{a}}^{\t}\times_{\Delta^{n}}\{0\}}\star\Delta^{\dim\sigma_{a}}\amalg_{\pr{\cal M_{/\sigma_{a}}^{\t}\times_{\Delta^{n}}\{0\}}\star\partial\Delta^{\dim\sigma_{a}}}\cal M_{<a}^{\t}\to\cal M_{\leq a}^{\t}
\]
is a trivial cofibration in the Joyal model structure. Thus we only
need to extend the composite
\[
g_{0}:\pr{\cal M_{/\sigma_{a}}^{\t}\times_{\Delta^{n}}\{0\}}\star\partial\Delta^{\dim\sigma_{a}}\to\cal M_{<a}^{\t}\xrightarrow{f^{<a}}\cal C^{\t}
\]
to a map $\pr{\cal M_{/\sigma_{a}}^{\t}\times_{\Delta^{n}}\{0\}}\star\Delta^{\dim\sigma_{a}}\to\cal C^{\t}$
over $\cal O^{\t}$.

Assume first that $\sigma_{a}$ is zero-dimensional, so that, in particular,
$m=1$ and $n=1$ (because if $n>1$, then every zero-dimensional
simplex is incomplete). Let $B\in\cal M^{\t}$ be the image of $\sigma_{a}$.
Since $\sigma_{a}$ is closed, the object $B$ lies in $\cal M\times_{\Delta^{1}}\{1\}$.
Using the inert-active factorization system in $\cal M^{\t}$, we
see that the inclusion $\pr{\pr{\cal M_{\act}^{\t}}_{/B}\times_{\Delta^{1}}\{0\}}\subset\cal M_{/B}^{\t}\times_{\Delta^{1}}\{0\}$
is a right adjoint, hence final. Thus, using condition (b), we can
find the desired extension of $g_{0}$. Note that condition (i) is
satisfied with this particular construction.

Assume next that $\sigma_{a}$ has positive dimension. Since $\Delta^{\dim\sigma_{a}}$
has an initial vertex, we see as in the previous paragraph that the
inclusion $\pr{\cal M_{\act}^{\t}}_{/\sigma_{a}}\times_{\Delta^{n}}\{0\}\subset\cal M_{/\sigma_{a}}^{\t}\times_{\Delta^{n}}\{0\}$
is a right adjoint, and hence final. Thus, by virtue of Proposition
\ref{prop:3.1.1.7}, it suffices to show that the map
\[
\pr{\pr{\cal M_{\act}^{\t}}_{/\sigma_{a}}\times_{\Delta^{n}}\{0\}}\star\{0\}\to\cal M_{<a}^{\t}\xrightarrow{f^{<a}}\cal C^{\t}
\]
is an operadic $q$-colimit diagram. Let $B=\sigma_{a}\pr 0$. The
map $\pr{\cal M_{\act}^{\t}}_{/\sigma_{a}}\times_{\Delta^{n}}\{0\}\to\pr{\cal M_{\act}^{\t}}_{/B}\times_{\Delta^{n}}\{0\}$
is a trivial fibration, so it suffices to show that the composite
\[
\phi_{B}:\pr{\pr{\cal M_{\act}^{\t}}_{/B}\times_{\Delta^{n}}\{0\}}\star\{0\}\to\cal M_{<a}^{\t}\xrightarrow{f_{1}}\cal C^{\t}
\]
is an operadic $q$-colimit cone. Let $\inp k\in N\pr{\Fin_{\ast}}$
be the image of the object $B$. If $k=0$, this is clear, because
$\pr{\cal M_{\act}^{\t}}_{/B}\times_{\Delta^{n}}\{0\}$ is a contractible
Kan complex and for each object $\alpha:A\to B$ in $\pr{\cal M_{\act}^{\t}}_{/B}\times_{\Delta^{n}}\{0\}$,
the map $\phi\vert\{\alpha\}\star\{0\}$ is an equivalence in $\cal C^{\t}$
(since $\cal C_{\inp 0}^{\t}$ is a contractible Kan complex). If
$k=1$, the claim follows from (i) or (B). If $k>1$, choose for each
$1\leq i\leq k$ an inert map $B\to B_{i}$ in $\cal M_{\act}^{\t}\times_{\Delta^{n}}\{1\}$
over $\rho^{i}:\inp k\to\inp 1$. Note that since $k>1$, we have
$m\geq2$, so the map $f^{<a}$ is defined on $\cal M_{F\pr 1}^{\t}$.
We will prove the following lemma in Subsection \ref{subsec:contribution1}:
\begin{lem}
\label{lem:contribution1}There is a categorical equivalence
\[
\theta:\prod_{i=1}^{k}\pr{\pr{\cal M_{\act}^{\t}}_{/B_{i}}\times_{\Delta^{n}}\{0\}}\xrightarrow{\simeq}\pr{\cal M_{\act}^{\t}}_{/B}\times_{\Delta^{n}}\{0\}
\]
such that the diagram
\[\begin{tikzcd}
	{(\prod_{i=1}^k((\mathcal{M}^{\otimes}_{\mathrm{act}/B_i})\times _{\Delta^n}\{0\}))^\triangleright} & {(\mathcal{M}^{\otimes}_{\mathrm{act}/B}\times _{\Delta^n}\{0\})^\triangleright} \\
	{\prod_{i=1}^k\mathcal{C}^{\otimes}_{\mathrm{act}}} & {\mathcal{C}^\otimes_{\mathrm{act}}}
	\arrow["{\theta^\triangleright}", from=1-1, to=1-2]
	\arrow["\simeq"', from=1-1, to=1-2]
	\arrow["{\prod_{i=1}^k\phi_{B_i}}"', from=1-1, to=2-1]
	\arrow["{\phi_B}", from=1-2, to=2-2]
	\arrow["{\bigoplus_{i=1}^k}"', from=2-1, to=2-2]
\end{tikzcd}\]commutes up to natural equivalence, where for each $1\leq i\leq k$,
the map $\phi_{B_{i}}$ denotes the (codomain restriction of the)
composite
\[
\pr{\cal M_{\act/B_{i}}^{\t}\times_{\Delta^{n}}\{0\}}^{\rcone}\to\cal M_{F\pr 1}^{\t}\xrightarrow{f_{1}}\cal C^{\t}.
\]
\end{lem}

Since each $\phi_{B_{i}}$ is an operadic $q$-colimit diagram by
(i) or (B), Lemma \ref{lem:contribution1} and Proposition \ref{prop:3.1.1.8}
show that the diagram $\phi_{B}$ is an operadic $q$-colimit diagram,
as claimed.

\item We will extend the map $f'_{m}$ in Step 1 to a map $f''_{m}:\cal M_{F''\pr m}^{\t}\to\cal C^{\t}$
over $\cal O^{\t}$. To understand why this extension is feasible,
recall that $F''\pr m$ is obtained from $F'\pr m$ by adjoining the
$\pr{m-1}$-simplices in $G'_{\pr 2}$ without associates (i.e., terminating
at the object $\pr{\inp 0,n}$). So the relevant extension problems
resemble the one we encountered in Corollary \ref{cor:Step2}, and
this corollary is exactly what we will use.

Here is the actual construction of the extension. We argue as in Step
1. Let $\{\sigma_{a}\}_{a\in A}$ be the set of all nondegenerate
simplices of $\cal M^{\t}\times_{\Delta^{n}}\Delta^{\{1,\dots,n\}}$
whose image in $N\pr{\Fin_{\ast}}\times\Delta^{\{1,\dots,n\}}$ is
a degeneration of an $\pr{m-1}$-simplex in $G'_{\pr 2}$ without
associates. Choose a well-ordering of the set $A$ so that $\dim\sigma_{a}$
is a non-decreasing function of $a$. For each $a\in A$, let $\cal M_{<a}^{\t}$
denote the simplicial subset spanned by $\cal M_{F'\pr m}^{\t}$ and
the simplices of $\cal M^{\t}$ whose head factors through $\sigma_{b}$
for some $b<a$. We define $\cal M_{\leq a}^{\t}$ similarly. (The
notations $\cal M_{<a}^{\t}$ and $\cal M_{\leq a}^{\t}$ are in conflict
with the ones introduced in Step 1, but there should be no confusion.)
By Proposition \ref{prop:heads_and_tails}, we have $\cal M_{F''\pr m}^{\t}=\cal M_{F'\pr m}^{\t}\cup\bigcup_{a\in A}\cal M_{\leq a}^{\t}$,
so it suffices to extend $f'_{m}$ to an $A$-sequence $f^{\leq a}:\cal M_{\leq a}^{\t}\to\cal C^{\t}$
over $\cal O^{\t}$. 

The construction is inductive. Suppose $f^{\leq b}$ has been constructed
for $b<a$, and let $f^{<a}:\cal M_{<a}^{\t}\to\cal C^{\t}$ be their
amalgamation. Just as in Step 1, we are reduced to solving a lifting
problem of the form 
\[\begin{tikzcd}
	{(\mathcal{M}^\otimes _{/\sigma_a}\times_{\Delta^n} \{0\})\star \partial\Delta ^{\dim\sigma _a}} & {\mathcal{M}^\otimes _{<a}} & {\mathcal{C}^\otimes} \\
	{(\mathcal{M}^\otimes _{/\sigma_a}\times_{\Delta^n} \{0\})\star \Delta ^{\dim\sigma _a}} & {\mathcal{M}^\otimes } & {\mathcal{O}^\otimes.}
	\arrow[from=1-1, to=2-1]
	\arrow[from=1-1, to=1-2]
	\arrow["{f^{<a}}", from=1-2, to=1-3]
	\arrow[from=2-1, to=2-2]
	\arrow[from=2-2, to=2-3]
	\arrow["q", from=1-3, to=2-3]
	\arrow[dashed, from=2-1, to=1-3]
\end{tikzcd}\]The existence of such a lift follows from Corollary \ref{cor:Step2}.

\item We complete the proof by extending the map $f''_{m}$ in Step
2 to a map $f_{m}:\cal M_{F\pr m}^{\t}\to\cal C^{\t}$ over $\cal O^{\t}$.
Recall that $F\pr m$ is obtained from $F''\pr m$ by adjoining the
$m$-simplices in $G_{\pr 2}$ and $G_{\pr 3}$. Let $\{\sigma'_{a}\}_{a\in A}$
be the collection of all $\pr{m-1}$-simplices in $G'_{\pr 2}$ and
$G'_{\pr 3}$ having associates. The set of associates of $\{\sigma'_{a}\}_{a\in A}$
is precisely $G_{\pr 2}\cup G_{\pr 3}$, and we base our extension
upon this observation. 

As before, our construction will be inductive. For this, we need a
nice ordering on $A$. The following lemma, which we will prove in
Subsection \ref{subsec:contribution2}, accomplishes this:
\begin{lem}
\label{lem:contribution2}There is a well-ordering on $A$ satisfying
the following condition:
\begin{itemize}
\item [($\blacklozenge$)]Let $a\in A$ and let $\sigma$ be an associate
of $\sigma'_{a}$. Let $0\leq l\leq m$ be the (unique) integer such
that $d_{l}\sigma=\sigma'_{a}$. Then for each $i\in[m]\setminus\{l\}$,
the simplex $d_{i}\sigma$ belongs to $F_{<a}$, which is defined
right after Remark \ref{rem:contribution2}.
\end{itemize}
Moreover, condition ($\blacklozenge$) implies that:
\begin{itemize}
\item [($\blacklozenge\blacklozenge$)]For every $a\in A$, the simplex
$\sigma'_{a}$ does not belong to $F_{<a}$.
\end{itemize}
\end{lem}

\begin{rem}
\label{rem:contribution2}In \cite{HA}, Lurie gives an explicit well-ordering
on $A$ satisfying condition ($\blacklozenge$), but never explains
why the ordering is suited for the purpose of the proof. In fact,
there is no mention of conditions ($\blacklozenge$) nor ($\blacklozenge\blacklozenge$)
in \cite{HA}. As these condition are essential to make the inductive
argument work (and the explicit well-ordering is less relevant to
the proof), we decided to state these conditions explicitly.
\end{rem}

Now choose a well-ordering on $A$ satisfying condition ($\blacklozenge$).
For each $a\in A$, let $F_{\leq a}$ denote the simplicial subset
of $F\pr m$ generated by $F''\pr m$ and the associates of the simplices
$\sigma'_{b}$ for $b\leq a$. Define $F_{<a}$ similarly. We will
construct $f_{m}$ as an amalgamation of an $A$-sequence $\{f_{\leq a}:\cal M_{F_{\leq a}}^{\t}\to\cal C^{\t}\}_{a\in A}$
over $\cal O^{\t}$ which extends $f''_{m}$. The construction is
inductive. Let $a\in A$, and suppose that $f_{\leq b}$ has been
constructed for $b<a$, so that they together determine a map $f_{<a}:\cal M_{F_{<a}}^{\t}\to\cal C^{\t}$.
We must extend $f_{<a}$ to $\cal M_{F_{\leq a}}^{\t}$. We consider
two cases, depending on whether $\sigma'_{a}$ belongs to $G'_{\pr 2}$
or to $G'_{\pr 3}$. 

\begin{enumerate}[label=(Case \arabic*),wide =0.5\parindent=, listparindent=1.5em]

\item Suppose that $\sigma'_{a}$ belongs to $G'_{\pr 2}$. Before
getting down to constructing the extension, let us outline why the
extension is possible. Let $\pr{\inp k,n}\in N\pr{\Fin_{\ast}}\times\Delta^{n}$
denote the final vertex of $\sigma'_{a}$. Note that $k\geq2$ because
$\sigma'_{a}$ is open and has an associate. There are $k$ associates
of $\sigma'_{a}$, determined by the $k$ inert maps $\inp k\to\inp 1$.
Roughly speaking, the extension will be possible because inert maps
of $\cal C^{\t}$ over these maps always exist and enjoy the universal
property of relative limits (by the definition of $\infty$-operads).

Now we get to the actual construction. We shall deploy an argument
which is a variant of Proposition \ref{prop:heads_and_tails}. Let
$\{\tau_{\lambda}\}_{\lambda\in\Lambda}$ be the collection of all
nondegenerate simplices of $\cal M^{\t}$ the image of whose head
in $N\pr{\Fin_{\ast}}\times\Delta^{\{1,\dots,n\}}$ is a degeneration
of $\sigma'_{a}$. Choose a well-ordering of $\Lambda$ such that
$\dim\tau_{\lambda}$ is a non-decreasing function of $\lambda$.
For each $\lambda\in\Lambda$, we let $\cal N_{\leq\lambda}\subset\cal M^{\t}$
denote the simplicial subset generated by $\cal M_{F_{<a}}^{\t}$
and the simplices $\tau:\Delta^{r}\to\cal M^{\t}$ for which there
is an integer $0\leq r'<r$ such that $\tau\vert\Delta^{\{0,\dots,r'\}}$
factors through some $\tau_{\mu}$ for some $\mu\leq\lambda$, $\tau\vert\Delta^{\{r',r'+1\}}$
is inert, and $\tau\vert\Delta^{\{r'+1,\dots,r\}}$ factors through
$\cal M$. We define $\cal N_{<\lambda}$ similarly. We will prove
the following assertion in Subsection \ref{subsec:contribution3}:
\begin{lem}
\label{lem:contribution3}The simplicial set $\cal M_{F_{\leq a}}^{\t}$
is the union of $\cal M_{F_{<a}}^{\t}$ and $\{\cal N_{\leq\lambda}\}_{\lambda\in\Lambda}$. 
\end{lem}

Accepting Lemma \ref{lem:contribution3} for now, we complete the
proof as follows. It will suffice construct a $\Lambda$-sequence
$\{f^{\leq\lambda}:\cal N_{\leq\lambda}\to\cal C^{\t}\}_{\lambda\in\Lambda}$
of maps over $\cal O^{\t}$ which extends $f_{<a}$. The construction
is inductive. Let $\lambda\in\Lambda$, suppose $f^{\leq\mu}$ has
been constructed for $\mu<\lambda$, and let $f^{<\lambda}:\cal N_{<\lambda}\to\cal C^{\t}$
denote the map obtained by amalgamating the maps $\{f^{\leq\mu}\}_{\mu<\lambda}$.
There are $k$ inert maps  $\inp k\to\inp 1$, and these maps and
$p\tau_{\lambda}$ combine to determine a diagram $\Delta^{\dim\tau_{\lambda}}\star\inp k^{\circ}\to N\pr{\Fin_{\ast}}\times\Delta^{n}$.
Set $\cal X=\pr{\Delta^{\dim\tau_{\lambda}}\star\inp k^{\circ}}\times_{N\pr{\Fin_{\ast}}\times\Delta^{n}}\cal M^{\t}$,
and let $\overline{\tau}_{\lambda}:\Delta^{\dim\tau_{\lambda}}\to\cal X$
denote the induced diagram. For each $1\leq i\leq k$, let $\cal X_{i}$
denote the fiber of $\cal X$ over $i\in\inp k^{\circ}$ (which is
isomorphic to $\cal M\times_{\Delta^{n}}\{n\}$), and set $\cal X^{0}=\bigcup_{1\leq i\leq k}\cal X_{i}$
and $\cal X_{\overline{\tau}_{\lambda}/}^{0}=\cal X^{0}\times_{\cal X}\cal X_{\overline{\tau}_{\lambda}/}$.
We now consider the following diagram: 
\begin{equation}\label{d:3.1.2.3}
\begin{tikzcd}
	{\partial\Delta^{\dim\overline\tau_\lambda}\star\mathcal{X}^0_{\overline\tau_\lambda/}} & {K_0} & {\mathcal{N}_{<\lambda}} \\
	{\Delta^{\dim\overline\tau_\lambda}\star\mathcal{X}^0_{\overline\tau_\lambda/}} & K & {\mathcal{N}_{\leq \lambda}.}
	\arrow[from=1-1, to=2-1]
	\arrow[from=1-1, to=1-2]
	\arrow[from=2-1, to=2-2]
	\arrow[from=1-2, to=2-2]
	\arrow[from=1-2, to=1-3]
	\arrow[from=2-2, to=2-3]
	\arrow[from=1-3, to=2-3]
\end{tikzcd}
\end{equation}Here $K\subset\cal X$ denotes the simplicial subset spanned by the
simplices whose tail factors through $\overline{\tau}_{\lambda}$,
where we regard $\cal X$ as a simplicial set over $\Delta^{1}$ by
the map $\Delta^{\dim\tau_{\lambda}}\star\inp k^{\circ}\to\{0\}\star\{1\}=\Delta^{1}$.
The simplicial set $K_{0}$ is defined similarly by replacing $\overline{\tau}_{\lambda}$
by $\partial\overline{\tau}_{\lambda}$ in the definition of $K$.
The left hand square is homotopy cocartesian by Lemma \ref{lem:3.1.2.5}.
In Subsection \ref{subsec:contribution4}, we will see that:
\begin{lem}
\label{lem:contribution4}The horizontal arrows of the the right hand
square of diagram (\ref{d:3.1.2.3}) are well-defined, and the right
hand square of the diagram (\ref{d:3.1.2.3}) is cocartesian.
\end{lem}

Using Lemma \ref{lem:contribution4}, we are reduced to solving the
lifting problem 
\[\begin{tikzcd}
	{\partial\Delta^{\dim\overline\tau_\lambda}\star\mathcal{X}^0_{\overline\tau_\lambda/}} & {\mathcal{N}_{<\lambda}} & {\mathcal{C}^\otimes } \\
	{\Delta^{\dim\overline\tau_\lambda}\star\mathcal{X}^0_{\overline\tau_\lambda/}} & {\mathcal{N}_{\leq \lambda}} & {\mathcal{O}^\otimes .}
	\arrow[from=1-1, to=2-1]
	\arrow[from=1-2, to=1-3]
	\arrow[from=2-2, to=2-3]
	\arrow[from=1-3, to=2-3]
	\arrow[dashed, from=2-1, to=1-3]
	\arrow[from=1-1, to=1-2]
	\arrow[from=2-1, to=2-2]
\end{tikzcd}\]Now the $\infty$-category $\cal X_{\overline{\tau}_{\lambda}/}^{0}$
is the disjoint union of the $\infty$-categories $\pr{\cal X_{i}}_{\overline{\tau}_{\lambda}/}=\cal X_{i}\times_{\cal X}\cal X_{\overline{\tau}_{\lambda}/}$.
Each $\infty$-category $\pr{\cal X_{i}}_{\overline{\tau}_{\lambda}/}$
has an initial object, given by a cone $\phi_{i}:\pr{\Delta^{\dim\overline{\tau}_{\lambda}}}^{\rcone}\to\cal X$
which maps the last edge to an inert morphism over $\pr{\rho^{i},\id}:\pr{\inp k,n}\to\pr{\inp 1,n}$.
Set $S=\coprod_{i}\{\phi_{i}\}$. The inclusion $S\subset\cal X_{\overline{\tau}_{\lambda}/}^{0}$
is initial, so we are reduced to solving the lifting problem 
\[\begin{tikzcd}
	{\partial\Delta^{\dim\overline\tau_\lambda}\star S} & {\partial\Delta^{\dim\overline\tau_\lambda}\star\mathcal{X}^0_{\overline\tau_\lambda/}} & {\mathcal{N}_{<\lambda}} & {\mathcal{C}^\otimes } \\
	{\Delta^{\dim\overline\tau_\lambda}\star S} & {\Delta^{\dim\overline\tau_\lambda}\star\mathcal{X}^0_{\overline\tau_\lambda/}} & {\mathcal{N}_{\leq \lambda}} & {\mathcal{O}^\otimes .}
	\arrow[from=1-3, to=1-4]
	\arrow[from=2-3, to=2-4]
	\arrow[from=1-4, to=2-4]
	\arrow[from=1-2, to=1-3]
	\arrow[from=2-2, to=2-3]
	\arrow[from=1-1, to=1-2]
	\arrow[from=2-1, to=2-2]
	\arrow[dashed, from=2-1, to=1-4]
	\arrow[from=1-1, to=2-1]
\end{tikzcd}\]

If the dimension of $\overline{\tau}_{\lambda}$ is positive, then
the lifting problem is trivial since the restriction of the top horizontal
arrow to $\{\dim\overline{\tau}_{\lambda}\}\star S$ is a $q$-limit
cone. If $\overline{\tau}_{\lambda}$ is zero-dimensional (in which
case $m=n=1$), let $C_{i}$ denote the image of $\phi_{i}\in S$
under the top horizontal map. The bottom horizontal arrow classifies
a diagram of inert maps $\{\alpha_{i}:X\to q\pr{C_{i}}\}_{1\leq i\leq k}$
lying over $\{\rho^{i}:\inp k\to\inp 1\}_{1\leq i\leq k}$, and we
wish to lift this to a diagram $S^{\lcone}\to\cal C^{\t}$ in $\cal C^{\t}$
which maps each $\phi_{i}\in S$ to the object $C_{i}$. Since $q$
is a fibration of $\infty$-operads, we can in fact find such a lift
consisting of inert morphisms. Note that with such a choice of lift,
condition (ii) is satisfied.

\item Suppose that $\sigma'_{a}$ belongs to $G'_{\pr 3}$. We will
show that the inclusion $\cal M_{<a}^{\t}\hookrightarrow\cal M_{\leq a}^{\t}$
is a weak categorical equivalence. The desired extension of $f_{<a}$
can then be found because the map $q$ is a categorical fibration.

Let $\sigma$ be the unique associate of $\sigma'_{a}$, which we
depict as 
\[
\pr{\inp{k_{0}},e_{0}}\xrightarrow{\alpha_{\sigma}\pr 1}\cdots\xrightarrow{\alpha_{\sigma}\pr m}\pr{\inp{k_{m}},e_{m}}.
\]
Let $0<j<k\leq m$ be the integers such that $\alpha_{\sigma}\pr i$
is strongly inert for $i=j$, active for $j<i\leq k$, and strongly
inert for $i>k$. Set $Y=\pr{N\pr{\Fin_{\ast}}\times\Delta^{n}}_{/\sigma}\times_{\Delta^{n}}\{0\}$.
Using ($\blacklozenge$), we may consider the following commutative
diagram: 
\[\begin{tikzcd}
	{Y\star \Lambda ^m_j} & {\overline{F_{<a}}} \\
	{Y\star \Delta ^m} & {\overline{F_{\leq a}},}
	\arrow[from=1-1, to=2-1]
	\arrow[from=1-1, to=1-2]
	\arrow[from=2-1, to=2-2]
	\arrow[from=1-2, to=2-2]
\end{tikzcd}\]where $\overline{F_{\leq a}}$ denotes the simplicial subset of $N\pr{\Fin_{\ast}}\times\Delta^{n}$
spanned by those simplices whose head belongs to $F_{\leq a}$, and
$\overline{F_{<a}}$ is defined similarly. Using ($\blacklozenge\blacklozenge$),
we deduce that this square is cocartesian. Therefore, it suffices
to show that the inclusion 
\[
\pr{Y\star\Lambda_{j}^{m}}\times_{N\pr{\Fin_{\ast}}\times\Delta^{n}}\cal M^{\t}\to\pr{Y\star\Delta^{m}}\times_{N\pr{\Fin_{\ast}}\times\Delta^{n}}\cal M^{\t}
\]
is a weak categorical equivalence. We will prove more generally that
for any morphism $Y'\to Y$ of simplicial sets, the map 
\[
\eta_{Y'}:\pr{Y'\star\Lambda_{j}^{m}}\times_{N\pr{\Fin_{\ast}}\times\Delta^{n}}\cal M^{\t}\to\pr{Y'\star\Delta^{m}}\times_{N\pr{\Fin_{\ast}}\times\Delta^{n}}\cal M^{\t}
\]
is a weak categorical equivalence. 

The assignment $Y'\mapsto\eta_{Y'}$ defines a functor from $\SS_{/Y}$
to the arrow category $\SS^{[1]}$ which commutes with filtered colimits.
Since weak categorical equivalences are stable under filtered colimits,
we may assume that $Y'$ is a finite simplicial set. If $Y'$ is empty,
the claim follows from Lemma \ref{lem:2.4.4.6}. If $Y'$ is nonempty,
we can find (by induction on the dimension of $Y'$ and the number
of nondegenerate simplices of $Y'$) a pushout diagram 
\[\begin{tikzcd}
	{\partial\Delta^p} & X \\
	{\Delta^p} & {Y'}
	\arrow[from=1-1, to=1-2]
	\arrow[from=1-2, to=2-2]
	\arrow[from=1-1, to=2-1]
	\arrow[from=2-1, to=2-2]
\end{tikzcd}\]in $\SS_{/Y}$, such that $\eta_{X}$ is a weak categorical equivalence.
The map $\eta_{Y'}$ factors as
\begin{align*}
\pr{Y'\star\Lambda_{j}^{m}}\times_{N\pr{\Fin_{\ast}}\times\Delta^{n}}\cal M^{\t} & \xrightarrow{\phi}\pr{Y'\star\Lambda_{j}^{m}\cup X\star\Delta^{m}}\times_{N\pr{\Fin_{\ast}}\times\Delta^{n}}\cal M^{\t}\\
 & \xrightarrow{\psi}\pr{Y'\star\Delta^{m}}\times_{N\pr{\Fin_{\ast}}\times\Delta^{n}}\cal M^{\t}.
\end{align*}
The map $\phi$ is a pushout of $\eta_{X}$, and hence is a trivial
cofibration. The map $\psi$ is a pushout of the inclusion
\[
\psi':\pr{\Delta^{p}\star\Lambda_{j}^{m}\cup\partial\Delta^{p}\star\Delta^{m}}\times_{N\pr{\Fin_{\ast}}\times\Delta^{n}}\cal M^{\t}\to\pr{\Delta^{p}\star\Delta^{m}}\times_{N\pr{\Fin_{\ast}}\times\Delta^{n}}\cal M^{\t}.
\]
Using the isomorphism of simplicial sets $\Delta^{p}\star\Lambda_{j}^{m}\cup\partial\Delta^{p}\star\Delta^{m}\cong\Lambda_{p+j+1}^{m+p+1}$
and Lemma \ref{lem:2.4.4.6}, we deduce that $\psi'$ is a weak categorical
equivalence. Hence $\psi$ is a weak categorical equivalence, completing
the treatment of Case 2.

\end{enumerate} 

\end{enumerate}
\end{proof}

\section{\label{sec:leftover}Leftover Proofs}

We now give proofs to the lemmas that appeared in Subsection \ref{subsec:Lurie's_proof}. 

\subsection{\label{subsec:contribution1}Proof of Lemma \ref{lem:contribution1}}

Using Proposition \ref{prop:oplus_p-limit}, choose a direct sum functor
$\bigoplus_{i=1}^{k}:\pr{\cal M_{\act}^{\t}}^{k}\times_{\pr{\Delta^{n}}^{k}}\Delta^{n}\to\cal M^{\t}$
so that, for each $1\leq i\leq k$, there is an inert natural transformation
$\widetilde{h}_{i}:\bigoplus_{i=1}^{k}\to\opn{pr}_{i}$ making the
diagram 
\[\begin{tikzcd}
	{((\mathcal{M}^\otimes_{\mathrm{act}})^k\times _{(\Delta^n)^k}\Delta^n)\times \Delta ^1} & {\mathcal{M}^\otimes } \\
	{(\Delta^n\times N(\mathsf{Fin}_\ast )_{\mathrm{act}})^k)\times \Delta^1} & {\Delta^n\times N(\mathsf{Fin}_\ast)}
	\arrow["{\widetilde{h}_i}", from=1-1, to=1-2]
	\arrow[from=1-1, to=2-1]
	\arrow["p", from=1-2, to=2-2]
	\arrow["{\operatorname{id}_{\Delta^n}\times h_i}"', from=2-1, to=2-2]
\end{tikzcd}\]commutative. The inert maps $\{B\to B_{i}\}_{1\leq i\leq k}$ determine
an equivalence $B\xrightarrow{\simeq}\bigoplus_{i=1}^{k}B_{i}$ that
lies over the identity morphism of $\pr{n,\inp k}$. Extend this equivalence
to a natural equivalence $\alpha:F\xrightarrow{\simeq}\bigoplus_{i=1}^{k}$
of functors $\pr{\cal M^{\t}}^{k}\times_{\pr{\Delta^{n}}^{k}}\Delta^{n}\to\cal M^{\t}$
over $\Delta^{n}\times N\pr{\Fin_{\ast}}$. For each $1\leq i\leq k$,
there is a commutative diagram
\[\begin{tikzcd}[column sep = small, scale cd=.8]
	{((\mathcal{M}^\otimes_{\mathrm{act}})^k\times _{(\Delta^n)^k}\Delta^n)\times \Lambda ^2 _1} &&& {\mathcal{M}^\otimes } \\
	{((\mathcal{M}^\otimes_{\mathrm{act}})^k\times _{(\Delta^n)^k}\Delta^n)\times \Delta^2} & {((\mathcal{M}^\otimes_{\mathrm{act}})^k\times _{(\Delta^n)^k}\Delta^n)\times \Delta^1} & {(\Delta^n\times N(\mathsf{Fin}_\ast )_{\mathrm{act}})^k)\times \Delta^1} & {\Delta^n\times N(\mathsf{Fin}_\ast)}
	\arrow["{(\alpha,\widetilde{h}_i)}", from=1-1, to=1-4]
	\arrow[from=1-1, to=2-1]
	\arrow["p", from=1-4, to=2-4]
	\arrow[dashed, from=2-1, to=1-4]
	\arrow["{\operatorname{id}\times s_1}"', from=2-1, to=2-2]
	\arrow[from=2-2, to=2-3]
	\arrow["{\operatorname{id}_{\Delta^n}\times h_i}"', from=2-3, to=2-4]
\end{tikzcd}\]which admits a dashed filler because the left vertical arrow is a
weak categorical equivalence. Thus, by replacing $\bigoplus_{i=1}^{k}$
with $F$ and $\widetilde{h}_{i}$ with the restriction of the filler
to $\pr{\cal M_{\act}^{\t}}^{k}\times_{\pr{\Delta^{n}}^{k}}\Delta^{n}\times\Delta^{\{1,2\}}$,
we may assume that $B=\bigoplus_{i=1}^{k}B_{i}$. Proposition \ref{prop:directsum_slice_equiv}
then gives us a categorical equivalence
\[
\theta:\pr{\pr{\cal M_{\act}^{\t}}^{k}\times_{\pr{\Delta^{n}}^{k}}\Delta^{n}}_{/\pr{B_{1},\dots,B_{k}}}\times_{\Delta^{n}}\{0\}\xrightarrow{\simeq}\pr{\cal M_{\act}^{\t}}_{/B}\times_{\Delta^{n}}\{0\}.
\]
We claim that $\theta$ has the desired property.

Consider the commutative diagram 
\[\begin{tikzcd}
	{(((\mathcal{M}^\otimes _{\mathrm{act}})^k\times _{(\Delta^n)^k}\Delta^n)_{/(B_1,\dots ,B_k)}\times _{\Delta^n}\{0\})^\triangleright} & {((\mathcal{M}^\otimes _{\mathrm{act}})_{/\bigoplus_{i=1}^kB_i}\times _{\Delta^n}\{0\})^\triangleright} \\
	{(\mathcal{M}^\otimes _{\mathrm{act}})^k\times _{(\Delta^n)^k}\Delta^n} & {\mathcal{M}^\otimes .}
	\arrow["{\theta^\triangleright}", from=1-1, to=1-2]
	\arrow["\simeq"', from=1-1, to=1-2]
	\arrow[from=1-1, to=2-1]
	\arrow[from=1-2, to=2-2]
	\arrow["{\bigoplus_{i=1}^k}"', from=2-1, to=2-2]
\end{tikzcd}\]In light of the commutativity of this diagram, the composite
\begin{align*}
\eta:\pr{\prod_{1\leq i\leq k}\pr{\pr{\cal M_{\act}^{\t}}_{/B_{i}}\times_{\Delta^{n}}\{0\}}}^{\rcone} & \cong\pr{\pr{\pr{\cal M_{\act}^{\t}}^{k}\times_{\pr{\Delta^{n}}^{p}}\Delta^{n}}_{/\pr{B_{1},\dots,B_{k}}}\times_{\Delta^{n}}\{0\}}^{\rcone}\\
 & \to\pr{\cal M_{\act}^{\t}}^{k}\times_{\pr{\Delta^{n}}^{k}}\Delta^{n}\\
 & \xrightarrow{\bigoplus_{i=1}^{k}}\cal M^{\t}
\end{align*}
takes values in $\cal M_{F\pr 1}^{\t}$. The inert natural transformation
$\widetilde{h}_{i}$ induces an inert natural transformation from
$\eta$ to the composite 
\[
\eta_{i}:\pr{\prod_{1\leq i\leq k}\pr{\pr{\cal M_{\act}^{\t}}_{/B_{i}}\times_{\Delta^{n}}\{0\}}}^{\rcone}\xrightarrow{{\rm pr}_{i}^{\rcone}}\pr{\pr{\cal M_{\act}^{\t}}_{/B_{i}}\times_{\Delta^{n}}\{0\}}^{\rcone}\to\cal M^{\t}.
\]
Since $F\pr 1$ contains the $1$-simplices in $G_{\pr 2}$, this
natural transformation takes values in $\cal M_{F\pr 1}^{\t}$. Since
$f_{1}$ satisfies (ii), we deduce that the composite $\phi_{B}\theta^{\rcone}=f_{1}\eta$
admits an inert natural transformation $H_{i}$ to the composite $\phi_{B_{i}}{\rm pr}_{i}^{\rcone}=f_{1}\eta_{i}$,
such that for each vertex $v$ in $\pr{\prod_{1\leq i\leq k}\pr{\pr{\cal M_{\act}^{\t}}_{/B_{i}}\times_{\Delta^{n}}\{0\}}}^{\rcone}$,
the maps $\{H_{i}\pr v\}_{1\leq i\leq k}$ form a $q$-limit cone.
It follows from Corollary \ref{cor:oplus_p-limit} that $f_{1}\eta$
is naturally equivalent to $\bigoplus_{i=1}^{k}\circ\pr{\phi_{B_{i}}}_{i=1}^{k}$,
as claimed.

\subsection{\label{subsec:contribution2}Proof of Lemma \ref{lem:contribution2}}

First, we explain why condition ($\blacklozenge\blacklozenge$) follows
from ($\blacklozenge$). Let $a\in A$, and suppose that $\sigma'_{a}$
belongs to $F_{<a}$. Choose a minimal element $b\in A$ such that
$\sigma'_{a}$ belongs to $F_{<b}$. (In particular, $b\leq a$.)
Then $\sigma'_{a}$ factors through one of the following simplices:

\begin{enumerate}[label=(\arabic*)]

\item Incomplete simplices.

\item Nondegenerate simplices of dimensions less than $m-1$.

\item $\pr{m-1}$-simplices in $G_{\pr 1}\cup G_{\pr 2}\cup G_{\pr 3}$.

\item $\pr{m-1}$-simplices in $G'_{\pr 2}$ without associates.

\item Associates of $\sigma'_{c}$ for some $c<b$.

\end{enumerate}

Since $\sigma'_{a}$ is complete and nondegenerate and has dimension
$m-1$, the cases (1), (2), (3), and (4) are immediately ruled out.
We show that the case (5) is impossible by reasoning by contradiction.
Suppose that there are an index $c<b$ and an associate $\sigma$
of $\sigma'_{c}$ through which $\sigma'_{a}$ factors. We then have
$\sigma'_{a}=d_{i}\sigma$ for some $0\leq i\leq m$ for dimensional
reasons. We also have $\sigma'_{c}\neq\sigma'_{a}$ because $c<b\leq a$.
It follows from ($\blacklozenge$) that $\sigma'_{a}$ belongs to
$F_{<c}$, contrary to the minimality of $b$.

Next, we construct a well-ordering on $A$ satisfying condition ($\blacklozenge$).
The construction is due to Lurie. For each $a\in A$, define integers
$u_{\mathrm{neut}}\pr a,u_{\act}\pr a,u_{\mathrm{oc}}\pr a,u_{\mathrm{as}}\pr a$
as follows: Choose an associate $\sigma$ of $\sigma'_{a}$, and set
$\alpha_{\sigma}\pr i=\sigma\vert\Delta^{\{i-1,i\}}$ for $1\leq i\leq m$.
Then:
\begin{itemize}
\item $u_{\mathrm{neut}}\pr a$ is the number of integers $1\leq i\leq m$
such that $\alpha_{\sigma}\pr i$ is neutral.
\item $u_{\act}\pr a$ is the number of integers $1\leq i\leq m$ such that
$\alpha_{\sigma}\pr i$ is active.
\item $u_{\mathrm{oc}}\pr a$ is set equal to $0$ if $\sigma$ is closed,
and is set equal to $1$ if $\sigma$ is open.
\item $u_{\mathrm{as}}\pr a$ is the number of pairs of integers $1\leq i<j\leq m$
such that $\alpha_{\sigma}\pr i$ is active and $\alpha_{\sigma}\pr j$
is strictly inert.
\end{itemize}
We choose a well-ordering on $A$ so that the function 
\[
A\ni a\mapsto\pr{u_{\mathrm{neut}}\pr a,u_{\act}\pr a,u_{\mathrm{oc}}\pr a,u_{\mathrm{as}}\pr a}\in\bb Z^{4}
\]
is non-decreasing, where $\bb Z^{4}$ is equipped with the lexicographic
ordering. (Thus $\pr{n_{1},n_{2},n_{3},n_{4}}<\pr{n'_{1},n'_{2},n'_{3},n'_{4}}$
if and only if $n_{i}<n'_{i}$, where $i$ is the minimal integer
for which $n_{i}\neq n'_{i}$.) We will show that this ordering does
the job.

Let $a,\sigma,i$ be as in ($\blacklozenge$). We wish to show that
$d_{i}\sigma$ belongs to $F_{<a}$. If $d_{i}\sigma$ is incomplete
or degenerate, then it belongs to $F\pr{m-1}$ and we are done. So
assume that $d_{i}\sigma$ is complete and nondegenerate. Then $d_{i}\sigma$
belongs to (exactly) one of the sets $G_{\pr 1},G_{\pr 2},G'_{\pr 2},G_{\pr 3},G'_{\pr 3}$.
If $d_{i}\sigma$ belongs to $G_{\pr 1}\cup G_{\pr 2}\cup G_{\pr 3}$,
or if it belongs to $G'_{\pr 2}$ and has no associates, then it belongs
to $F''\pr m$ and we are done. So we will assume that $d_{i}\sigma$
belongs to $G'_{\pr 2}\cup G'_{\pr 3}$ and has an associate, so that
$d_{i}\sigma=\sigma'_{b}$ for some $b\in A$. We wish to show that
$b<a$.

We will make use of the following notations: For each $1\leq i\leq m$,
we set $\alpha_{\sigma}\pr i=\sigma\vert\Delta^{\{i-1,i\}}$. We let
$0\leq k\leq m$ denote the minimal integer such that $\alpha_{\sigma}\pr i$
is strongly inert for every $i>k$, and let $0\leq j\leq k$ denote
the minimal integer such that $\alpha_{\sigma}\pr i$ is active for
every $j<i\leq k$. 

Suppose first that $\sigma'_{a}\in G'_{\pr 2}$. Our assumption on
$d_{i}\sigma$ implies that $i=k$ and that the composite $\alpha_{\sigma}\pr{k+1}\circ\alpha_{\sigma}\pr k$
is neutral. It follows that the associate $\tau$ of $d_{i}\sigma$
satisfies $\alpha_{\tau}\pr s=\alpha_{\sigma}\pr s$ for $s\neq k,k+1$,
$\alpha_{\tau}\pr k$ is strictly inert, and $\alpha_{\tau}\pr{k+1}$
is active. Thus $u_{\mathrm{neut}}\pr a=u_{\mathrm{neut}}\pr b$,
$u_{\act}\pr a=u_{\act}\pr b$, $u_{\mathrm{oc}}\pr a=u_{\mathrm{oc}}\pr b$,
and $u_{\mathrm{as}}\pr a>u_{\mathrm{as}}\pr b$. Hence $a>b$, as
required.

Suppose next that $\sigma'_{a}\in G'_{\pr 3}$ and that $d_{i}\sigma\in G'_{\pr 2}$. 
\begin{itemize}
\item If $u_{\mathrm{neut}}\pr a>0$, we are done, since $u_{\mathrm{neut}}\pr b=0$.
\item If $u_{\mathrm{neut}}\pr a=0$, then each $\alpha_{\sigma}\pr i$
is either active or strongly inert. It follows that $u_{\act}\pr a$
is not less than the number of active morphisms in $d_{i}\sigma$,
which is equal to $u_{\act}\pr b$. So $u_{\act}\pr a\geq u_{\act}\pr b$.
If $u_{\act}\pr a>u_{\act}\pr b$, we are done. 
\item If $u_{\mathrm{neut}}\pr a=0$ and $u_{\act}\pr a=u_{\act}\pr b$,
then $\sigma$ is open. Indeed, if $\sigma$ were closed, then $i=m$
since $d_{i}\sigma$ is open, and $j<m$ since $u_{\act}\pr a=u_{\act}\pr b$.
But then $d_{i}\sigma$ must contain a subsequence consisting of a
strictly inert morphism followed by an active morphism, which is impossible
because $d_{i}\sigma$ belongs to $G'_{\pr 2}$. Hence $u_{\mathrm{oc}}\pr a>u_{\mathrm{oc}}\pr b$
and we are done.
\end{itemize}

Finally, suppose that $\sigma'_{a}\in G'_{\pr 3}$ and that $d_{i}\sigma\in G'_{\pr 3}$.
Note that our assumption on $d_{i}\sigma$ forces $i=j-1$ or $i=k+1$
and $0<i<m$.
\begin{itemize}
\item The number of neutral morphisms in $d_{i}\sigma$ is at most $u_{\mathrm{neut}}\pr a+1$,
so 
\begin{equation}
u_{\mathrm{neut}}\pr a\geq\#\{\text{neutral morphisms in }d_{i}\sigma\}-1=u_{\mathrm{neut}}\pr b.\label{eq:1}
\end{equation}
If $u_{\mathrm{neut}}\pr a>u_{\mathrm{neut}}\pr b$, we are done. 
\item Suppose $u_{\mathrm{neut}}\pr a=u_{\mathrm{neut}}\pr b$, so that
the equality holds in (\ref{eq:1}). Then the composite $\alpha_{\sigma}\pr{i+1}\circ\alpha_{\sigma}\pr i$
is neutral, $\alpha_{\sigma}\pr i$ is active, and $\alpha_{\sigma}\pr{i+1}$
is strictly inert. It follows that the associate $\tau$ of $d_{i}\sigma$
satisfies $\alpha_{\tau}\pr s=\alpha_{\sigma}\pr s$ for $s\neq i,i+1$,
$\alpha_{\tau}\pr i$ is strictly inert, and $\alpha_{\tau}\pr{i+1}$
is active. Thus $u_{\mathrm{neut}}\pr a=u_{\mathrm{neut}}\pr b$,
$u_{\act}\pr a=u_{\act}\pr b$, $u_{\mathrm{oc}}\pr a=u_{\mathrm{oc}}\pr b$,
and $u_{\mathrm{as}}\pr a>u_{\mathrm{as}}\pr b$. Hence $a>b$, as
desired.
\end{itemize}
This completes the proof that the ordering on $A$ satisfies condition
($\blacklozenge$).

\subsection{\label{subsec:contribution3}Proof of Lemma \ref{lem:contribution3}}

It is clear that $\cal N_{\leq\lambda}$ and $\cal M_{F_{<a}}^{\t}$
are contained in $\cal M_{F_{\leq a}}^{\t}$. For the reverse containment,
let $x$ be an arbitrary simplex of $\cal M_{F_{\leq a}}^{\t}$. We
must show that $x$ belongs to either $\cal M_{F_{<a}}^{\t}$ or one
of the $\cal N_{\leq\lambda}$'s. If $x$ belongs to $\cal M_{F_{<a}}^{\t}$,
we are done. So assume not. Let $r$ denote the dimension of $x$.
Since $x$ does not belong to $\cal M_{F_{<a}}^{\t}$, its head is
nonempty. Find an integer $0\leq r'\leq r$ such that $x\vert\Delta^{\{r',\dots,r\}}$
is the head of $x$. Since $x$ does not belong to $\cal M_{F_{<a}}^{\t}$,
the simplex $px\vert\Delta^{\{r',\dots,r\}}$ factors through an associate
$\sigma$ of $\sigma'_{a}$. Let $\Delta^{\{r',\dots,r\}}\xrightarrow{u}\Delta^{m}\xrightarrow{\sigma}N\pr{\Fin_{\ast}}\times\Delta^{\{1,\dots,n\}}$
be such a factorization. If the image of $u$ does not contain some
integer $i\in\{0,\dots,m-1\}$, then $px\vert\Delta^{\{r',\dots,r\}}$
factors through $d_{i}\sigma$ and hence through $F_{<a}$ by ($\blacklozenge$),
a contradiction. So the image of $u$ contains every integer in $\{0,\dots,m-1\}$.
Let $r'\leq r''\leq r$ be the largest integer such that $u\pr{r''}<m$.
There are now several cases to consider:
\begin{itemize}
\item Suppose that $u$ is surjective on vertices and that the map $x\pr{r''}\to x\pr{r''+1}$
is inert. We write $x\vert\Delta^{\{0,\dots,r''\}}=s^{*}y$, where
$s:[r'']\to[k]$ is a surjective poset map and $y$ is nondegenerate.
We claim that $y$ is one of the $\tau_{\lambda}$'s, so that $x$
belongs to $\cal N_{\leq\lambda}$. Let $k'=s\pr{r'}$. Then $y\vert\Delta^{\{k',\dots,k\}}$
is the head of $y$. We write $py\vert\Delta^{\{k',\dots,k\}}=s^{\p\ast}z$,
where $z$ is nondegenerate and $s'$ is a surjection. Then we obtain
the following commutative diagram: 
\[\begin{tikzcd}
	{\Delta^{\{0,\dots,r''\}}} & {\Delta^{\{r',\dots,r''\}}} & {\Delta^{m-1}} \\
	{\Delta^{k}} & {\mathcal{M}^\otimes } & {N(\mathsf{Fin}_\ast)\times \Delta^n.} \\
	{\Delta^{\{k',\dotsm,k\}}} & {\Delta^l}
	\arrow["{u\vert\Delta^{\{r',\dots,r''\}}}", two heads, from=1-2, to=1-3]
	\arrow["{\sigma_a'}", from=1-3, to=2-3]
	\arrow["p"', from=2-2, to=2-3]
	\arrow["s"', two heads, from=1-1, to=2-1]
	\arrow["y"', from=2-1, to=2-2]
	\arrow[hook', from=1-2, to=1-1]
	\arrow[hook', from=3-1, to=2-1]
	\arrow["{s'}"', two heads, from=3-1, to=3-2]
	\arrow[from=3-1, to=2-2]
	\arrow["z"', from=3-2, to=2-3]
	\arrow["{x\vert\Delta^{\{0,\dots,r''\}}}"{description}, from=1-1, to=2-2]
\end{tikzcd}\]The diagram shows that $px\vert\Delta^{\{r',\dots,r''\}}$ is a degeneration
of $\sigma'_{a}$ and $z$, both of which are nondegenerate. It follows
that $z=\sigma'_{a}$. Thus $y$ is one of the simplices in $\{\tau_{\lambda}\}_{\lambda\in\Lambda}$,
as required.
\item Suppose that $u$ is surjective on vertices and that the map $\theta:x\pr{r''}\to x\pr{r''+1}$
is not inert. By factoring the map $\theta$ into an inert map followed
by an active map, we can find an $\pr{r+1}$-simplex $y$ of $\cal M_{F_{\leq a}}^{\t}$
such that $d_{r''+1}y=x$ and $y\vert\Delta^{\{r'',r''+1,r''+2\}}$
is the chosen factorization of $\theta$. By the previous point, the
simplex $y$ belongs to $\cal N_{\leq\lambda}$ for some $\lambda$,
and hence so must $x$.
\item Suppose $u$ is not surjective on vertices. Find an inert map $\theta:x\pr r\to X$
over the last edge of $\sigma$, and let $y$ be an $\pr{r+1}$-simplex
$y$ of $\cal M_{F_{\leq a}}^{\t}$ such that $d_{r+1}y=x$ and $y\vert\Delta^{\{r,r+1\}}=\theta$.
By the first point, the simplex $y$ belongs to $\cal N_{\leq\lambda}$
for some $\lambda$, and hence so must $x$.
\end{itemize}
This completes the proof of Lemma \ref{lem:contribution3}.

\subsection{\label{subsec:contribution4}Proof of Lemma \ref{lem:contribution4}}

First we show that the maps $K\to\cal N_{\leq\lambda}$ and $K_{0}\to\cal N_{<\lambda}$
are well-defined. A typical simplex in the image of the map $K\to\cal M^{\t}$
has the form $x:\Delta^{r}\to\cal M^{\t}$, where there is an integer
$-1\leq r'\leq r$ such that $x\vert\Delta^{\{0,\dots,r'\}}$ factors
through $\tau_{\lambda}$ and, if $r'<r$, then $p\pr{x\vert\Delta^{\{r',r'+1\}}}$
is inert and and $p\pr{x\vert\Delta^{\{r'+1,\dots,r\}}}$ is the constant
map at $\pr{\inp 1,n}$. (When $r'=-1$, the symbol $\Delta^{\{0,\dots,r'\}}$
denotes the empty simplicial set.) We must show that such a simplex
belongs to $\cal M_{F_{<a}}^{\t}$or $\cal N_{\leq\lambda}$. There
are three cases to consider:
\begin{itemize}
\item If $r'=-1$, then $x$ lies in $\cal M_{\pr{\inp n,1}}^{\t}$. Since
the vertex $\pr{\inp n,1}\in N\pr{\Fin_{\ast}}\times\Delta^{n}$ belongs
to $F'\pr 0$ if $n=1$ and to $F\pr 0$ if $n>1$,  $x$ belongs
to $\cal M_{F_{<a}}^{\t}$.
\item If $0\leq r'<r$, then by factoring the map $x\pr{r'}\to x\pr{r'+1}$
into an inert map followed by an active map, we can find a simplex
$y$ belonging to $\cal N_{\leq\lambda}$ and satisfying $d_{r'+1}y=x$.
Hence $x$ belongs to $\cal N_{\leq\lambda}$.
\item If $r'=r$, then by choosing an inert map $x\pr{r'}\to X$ with $X\in\cal M$,
we can find a simplex $y$ belonging to $\cal N_{\leq\lambda}$ and
satisfying $d_{r'+1}y=x$. Hence $x$ belongs to $\cal N_{\leq\lambda}$.
\end{itemize}
Next, for the map $K_{0}\to\cal N_{<\lambda}$, assume further that
$x\vert\Delta^{\{0,\dots,r'\}}$ factors through the boundary of $\tau_{\lambda}$.
We must show that $x$ belongs to $\cal M_{F_{<a}}^{\t}$ or $\cal N_{<\lambda}$.
We have two cases to consider:
\begin{itemize}
\item If $r'=-1$, then $x$ belongs to $\cal M_{F_{<a}}^{\t}$ as before.
\item Suppose that $r'\geq0$ and that the head of the simplex $p\pr{x\vert\Delta^{\{0,\dots,r'\}}}$
factors through $\partial\sigma'_{a}$. Then the head of $p\pr x$
factors through $d_{i}\sigma$ for some $0\leq i<m$ for some associate
$\sigma$ of $\sigma'_{a}$, so $x$ belongs to $\cal M_{F_{<a}}^{\t}$
by ($\blacklozenge$). 
\item Suppose that $r'\geq0$ and that the head of the simplex $p\pr{x\vert\Delta^{\{0,\dots,r'\}}}$
is a degeneration of $\sigma'_{a}$.Write $x\vert\Delta^{\{0,\dots,r'\}}=s^{*}y$,
where $y$ is nondegenerate and $s$ is a surjection. Since the head
of $p\pr{x\vert\Delta^{\{0,\dots,r'\}}}$ is a degeneration of both
$\sigma'_{a}$ and the head of $p\pr y$, and since $\sigma'_{a}$
is nondegenerate, the head of $p\pr y$ is a degeneration of $\sigma'_{a}$.
It follows that $y=\tau_{\mu}$ for some $\mu\in\Lambda$. Since $x\vert\Delta^{\{0,\dots,r'\}}$
factors through the boundary of $\tau_{\lambda}$, the dimension of
$y=\tau_{\mu}$ must be smaller than that of $\tau_{\lambda}$. Thus
$\mu<\lambda$. The above argument (that the map $K\to\cal N_{\leq\lambda}$
is well-defined) shows that $x$ belongs to $\cal N_{\leq\mu}$, so
$x$ belongs to $\cal N_{<\lambda}$.
\end{itemize}
This completes the verification of the first half of Lemma \ref{subsec:contribution4}. 

We next proceed to the latter half. We must show that the square 
\[\begin{tikzcd}
	{K_0} & {\mathcal{N}_{<\lambda}} \\
	K & {\mathcal{N}_{\leq \lambda}}
	\arrow[from=1-1, to=2-1]
	\arrow[from=1-1, to=1-2]
	\arrow[from=2-1, to=2-2]
	\arrow[from=1-2, to=2-2]
\end{tikzcd}\]is cocartesian. We first show that $\cal N_{\leq\lambda}$ is the
union of the images of $K$ and $\cal N_{<\lambda}$. Unwinding the
definitions, we must show that if $x:\Delta^{r}\to\cal M^{\t}$ is
a simplex for which there is some integer $0\leq r'<r$ such that
$x\vert\Delta^{\{0,\dots,r'\}}$ factors through $\tau_{\lambda}$,
$x\vert\Delta^{\{r',r'+1\}}$ is inert, and $x\vert\Delta^{\{r'+1,\dots,r\}}$
factors through $\cal M$, then $x$ belongs to the union of the images
of $K$ and $\cal N_{\lambda}$. If $x\pr{r'}$ lies over the vertex
$n\in\Delta^{n}$, then $x$ belongs to the image of $K$. If $x\pr{r'}$
does not lie over the vertex $n\in\Delta^{n}$, then the head of $x$
is incomplete, so $x$ belongs to $\cal M_{F\pr 0}^{\t}$ and hence
to $\cal N_{<\lambda}$. Thus we have shown that $\cal N_{\leq\lambda}$
is the union of the images of $K$ and $\cal N_{<\lambda}$.

By the result in the previous paragraph, to show that the square is
cocartesian, it suffices to show that if a simplex $z$ of $K$ is
mapped into $\cal N_{<\lambda}$, then $z$ belongs to $K_{0}$. Taking
the contrapositive, we will show that if $z$ does not belong to $K_{0}$,
then its image in $\cal N_{\leq\lambda}$ does not belong to $\cal N_{<\lambda}$.
Let $x:\Delta^{r}\to\cal N_{\leq\lambda}$ be the image of $z$. By
the definition of the map $K\to\cal N_{\leq\lambda}$ and by the hypothesis
that $z$ does not belong to $K_{0}$, there is an integer $0\leq r'\leq r$
such that $x\vert\Delta^{\{0,\dots,r'\}}$ is a degeneration of $\tau_{\lambda}$
and, if $r'<r$, then $p\pr{x\vert\Delta^{\{r',r'+1\}}}$ is inert
and and $p\pr{x\vert\Delta^{\{r'+1,\dots,r\}}}$ is the constant map
at $\pr{\inp 1,n}$. The head of $p\pr x$ is thus a degeneration
of either $\sigma'_{a}$ or its associate, and so it does not belong
to $F_{<a}$ by ($\blacklozenge\blacklozenge$). Therefore, $x$ does
not belong to $\cal M_{F_{<a}}^{\t}$. So should $x$ belong to $\cal N_{<\lambda}$,
then $x\vert\Delta^{\{0,\dots,r'\}}$ must factor through $\tau_{\mu}$
for some $\mu<\lambda$. This is impossible because $x\vert\Delta^{\{0,\dots,r'\}}$
is a degeneration of $\tau_{\lambda}$. Hence $x$ does not belong
to $\cal N_{<\lambda}$. This completes the proof of Lemma \ref{lem:contribution4}.

\section*{Declarations}

\subsection*{Funding}

No funding was received for conducting this study.

\subsection*{Competing Interests}

The author has no financial or proprietary interests in any material
discussed in this article.

\input{HA_3.1.2.3.bbl}

\end{document}

%% file: HA_3.1.2.3.bbl
\providecommand{\bysame}{\leavevmode\hbox to3em{\hrulefill}\thinspace}
\providecommand{\MR}{\relax\ifhmode\unskip\space\fi MR }
\providecommand{\MRhref}[2]{%
  \href{http://www.ams.org/mathscinet-getitem?mr=#1}{#2}
}
\providecommand{\href}[2]{#2}